\documentclass{ws-ccm}

\usepackage{mathtools}
\usepackage{graphicx}

\newcommand{\carrenoir}{\nopagebreak\hfill$\blacksquare$\par\medskip}

\newcommand{\eps}{\varepsilon}

\newcommand{\R}{\mathbb R}
\newcommand{\C}{\mathbb{C}}
\newcommand{\NN}{\mathbb{N}}

\newcommand{\beq}{\begin{equation}}
\newcommand{\eeq}{\end{equation}}

\newcommand{\ds}{\displaystyle}

\newcommand{\ts}{\textstyle}
\newcommand{\ol}[1]{\overline{#1}}

\newcommand{\grad}{\nabla}

\newcommand{\wto}{\rightharpoonup}
\def\geq{\geqslant}
\def\leq{\leqslant}
\def\div{{\rm div}\,}
\def\curl{{\rm curl}\,}
\def\Hdiv{\breve{H}^1_{\div}}
\newcommand{\nnn}{\nonumber}

\newcommand{\HH}{\breve{H}_{\div}^1}
\newcommand{\HHr}{\HH(\R^3;\R^3)}
\DeclareMathOperator*{\argmin}{arg\, min}

\begin{document}

\markboth{S. Alama \& L. Bronsard \& B. Galv\~ao-Sousa} {Thin film limits for Ginzburg--Landau}

\title{Thin film limits for Ginzburg--Landau with strong applied magnetic fields}

\author{STAN ALAMA \& LIA BRONSARD \& BERNARDO GALV\~AO-SOUSA}

\address{Department of Mathematics and Statistics\\ McMaster University\\
Hamilton, ON, Canada
\\
\emph{\texttt{\{alama,bronsard,beni\}@mcmaster.ca}}}

\maketitle

\begin{abstract}
{\bf Abstract.} 
In this work, we study thin-film limits of the full three-dimensional Ginzburg--Landau model for a superconductor in an applied magnetic field oriented obliquely to the film surface.  We obtain $\Gamma-$convergence results in several regimes, determined by the asymptotic ratio between the magnitude of the parallel applied magnetic field and the thickness of the film. Depending on the regime, we show that there may be a decrease in the density of Cooper pairs. We also show that in the case of variable thickness of the film, its geometry will affect the effective applied magnetic field, thus influencing the position of vortices.
\end{abstract}


\ccode{Mathematics Subject Classification 2000: }

\section{Introduction}

In this paper we consider superconducting thin films subjected to an external magnetic field, using the Ginzburg--Landau model.  We assume the superconductor occupies a domain $\Omega_\eps\subset\R^3$ of variable but small thickness, which projects to a smooth planar domain $\omega\subset\R^2$,
$$ {\bf x}=({\bf x}',{\bf x}_3)\in \Omega_\eps \quad\iff \quad  {\bf x}'\subset\omega, \ 
      \eps f({\bf x}')<{\bf x}_3< \eps g({\bf x}'),  $$
for given smooth functions $f,g: \ \omega\to \R$ with $\inf_\omega (g-f)>0$.
Here, and throughout, we denote the projection of ${\bf x}\in\R^3$ to the plane by ${\bf x}'=({\bf x}_1,{\bf x}_2)\in\R^2$.  The state of the superconductor is described by
a complex-valued order parameter, ${\mathbf u}: \ \Omega_\eps\to \C$ defined inside the sample, and the magnetic vector potential ${\mathbf A}: \ \R^3\to \R^3$, which determines the magnetic field ${\mathbf h}={\mathbf\nabla}\times{\mathbf A}$.  We assume that the superconductor is placed in a constant magnitude, externally applied magnetic field ${\mathbf h}_\eps^{\rm ex}$, which may be oriented {\em obliquely} with respect to the plane of $\omega$.  With these choices, the Ginzburg--Landau energy functional is given by
\begin{equation*}
{\bf I}_{\kappa,\eps}({\bf u},{\bf A}) := \frac{1}{\eps}\biggl( \frac12 \int_{\Omega_{\eps}} \left( |\grad_A {\bf u} |^2 +  \frac{\kappa^2}{2} \bigl( 1 - |{\bf u}|^2 \bigr)^2 \right) \, d{\bf x} + \frac12 \int_{\R^3} |{\bf h} - {\bf h}^{\rm ex}_{\eps}|^2  \, d{\bf x} \biggr),
\end{equation*}
We note that the factor $1/\eps$ which multiplies the energy is not traditionally present, but is useful here since the energy of minimizers will be order-one with this normalization.

Motivated by recent work on the Lawrence--Doniach model (\cite{ABS1}, \cite{ABS3}) we are particularly interested in the behavior of the thin film superconductor in applied fields which are {\em parallel} (or nearly parallel) to the plane of $\omega$.  In order to see the effect of strong parallel fields, we allow the parallel component of the applied field 
${{\mathbf h}_\eps^{\rm ex}}' (\in\R^2)$ to depend on the thickness parameter $\eps$,
\begin{equation}\label{hex}
{\mathbf h}_\eps^{\rm ex} := (\rho_\eps {h^{\rm ex}}', h_3^{\rm ex}),
\end{equation}
We will identify different $\Gamma$--limits, in the sense of De Giorgi (see \cite{deGiorgi, GF, DalMaso, Braides}), depending on the magnitude of $\rho_\eps$.  The limiting behavior of minimizers of ${\mathbf I}_{\kappa,\eps}$ with applied fields of fixed magnitude ($\rho_\eps=1$) was studied by Chapman, Du \& Gunzburger \cite{CDG}.  By means of an asymptotic expansion using the Euler--Lagrange equations and estimates on the minimum energy they show that the vertical averages of the order parameters ${\mathbf u}_\eps$ and potentials ${\mathbf A}_\eps$ converge (weakly in $H^1$) to a solution of a simplified two-dimensional Ginzburg--Landau model, in which the limiting vector potential produces the vertical component ${\bf h}^{\rm ex}_3$ of the applied field.  Our results (below) reproduce this outcome as part of a more general $\Gamma$--convergence setting, in the appropriate (``subcritical'') regime.  The critical case, $\rho_\eps=O(\eps^{-1})$, and supercritical cases produce very different and interesting results, which we will describe below.

In preparing this manuscript we have learned of very recent work by Contreras \& Sternberg \cite{CS} on $\Gamma$-limits for thin film superconductors, but with a very different point of view.  They consider thin shells based on fixed closed manifolds in $\R^3$, with magnetic fields independent of $\eps$.

To identify the correct scales in the problem, we introduce the following rescaled coordinates:
\begin{align*}
& x = (x',x_3)= (x_1,x_2,x_3) = \bigg({\bf x}_1, {\bf x}_2, \frac{{\bf x}_3}{\eps}\bigg) \\
& A(x) = ({\bf A}_1, {\bf A}_2, \eps {\bf A}_3)\bigg({\bf x}_1, {\bf x}_2, \frac{{\bf x}_3}{\eps}\bigg), \\
& u(x) = {\bf u}({\bf x}).
\end{align*}
In the new coordinates, the magnetic field $h=\nabla\times A$ transforms in a straightforward way,
$$  {\mathbf h}={\mathbf \nabla}\times{\mathbf A} =      
     \left( {1\over\eps}(\partial_2 A_3-\partial_3 A_2), 
      {1\over\eps}(\partial_3 A_1-\partial_1 A_3),
        (\partial_1 A_2 -\partial_2 A_1)\right)=
              \left({1\over\eps} h', h_3\right),
$$
and similarly for ${\mathbf h}^{\rm ex}=  \left({1\over\eps} {h^{\rm ex}}', h^{\rm ex}_3\right).$  Note also that the divergence free condition
$\nabla\cdot h=0$ is preserved under this rescaling.

Denote the rescaled domain
$$  \Omega:=\Omega_1=\{ (x', x_3)\in\R^3: \  f(x')< x_3<g(x'), \quad
      x'\in\omega\}.  $$
Then, the Ginzburg--Landau energy becomes:
\begin{multline}  
I_{\kappa,\eps}(u,A) := 
	\frac12 \int_{\Omega} \left( | (\grad'-iA')u|^2 + \left| \frac{1}{\eps}(\partial_3 -iA_3)u \right|^2 + \frac{\kappa^2}{2} \bigl( 1 - |u|^2 \bigr)^2 \right) \, dx  \\
	+ \frac12 \int_{\R^3} \left(|h_3 - h^{\rm ex}_3|^2 + \frac{1}{\eps^2} \left| h' - \eps \rho_{\eps}{h^{\rm ex}}' \right|^2 \right)  \, dx.
	\label{GL}
\end{multline}
In keeping with our notation above, $\grad' = (\partial_1, \partial_2)$. 

We must also define function spaces for our configurations $(u,A)$.  This is complicated both by the fact that $A$ is defined in the whole space $\R^3$ and the gauge invariance of the energy.  The natural space for the order parameter is $u\in H^1(\Omega;\C)$.  To define a space for the vector potential $A$ we must essentially fix an appropriate gauge, which also captures the behavior of the field at infinity.  First, we fix a representative for the constant effective external field, $(\eps\rho_\eps h^{\rm ex}{}',h^{\rm ex}_3)$,
\begin{equation}\label{Aex}
 A^{\rm ex}_\eps= \frac12 (\eps\rho_\eps h^{\rm ex}{}',h^{\rm ex}_3)\times (x_1,x_2,x_3) =
      \frac12( \eps\rho_\eps h^{\rm ex}_2 x_3-h^{\rm ex}_3 x_2, h^{\rm ex}_3 x_1 - \eps\rho_\eps h^{\rm ex}_1 x_3, \eps\rho_\eps(h^{\rm ex}_1 x_2 - h^{\rm ex}_2 x_1)).
\end{equation}
Then, we assume $A-A^{\rm ex}_\eps\in \HH(\R^3;\R^3)$, defined as the completion of the space  of smooth, compactly supported, divergence free vector fields $C_0^\infty(\R^3;\R^3)$, in the Dirichlet norm, $\| F \|_{\HH}=[\int_{\R^n} |DF|^2\, dx]^{1/2}$.  (See Giorgi \& Phillips \cite{GP}.)  

\medskip

With the energy of the form \eqref{GL}, we may now identify the different limiting regimes as $\eps\to 0$.  We identify the {\em subcritical regime} with $\eps\rho_\eps \to 0$, the {\em critical regime} corresponds to $\eps\rho_\eps\to L\neq 0$, and $\eps\rho_\eps \to\infty$ in the {\em supercritical regime.}
We prove a $\Gamma$--convergence result for each regime:  Assume $\eps_n\to 0^+$ is any sequence, and $(u_n, A_n)$ with $u\in H^1(\Omega_1;\C)$ and
$A_n-A^{\rm ex}_{\eps_n}\in \HH(\R^3;\R^3)$ is a sequence with bounded energy $\sup_n I_{\kappa,\eps_n}(u_n,A_n)<\infty$.  

\subsection*{The critical regime}
By adjusting the constant values of $h^{\rm ex}{}'$, we may simplify our condition to $\eps \rho_\eps \to 1$,
and neglect the $\eps$ dependence of $A^{\rm ex}$.  This is the most interesting case, as it leads to two new phenomena in the limiting energy.  

First, we obtain a compactness result:  there exists $v\in H^1(\omega;\C)$
and $b\in L^2(\omega;\C)$ so that
\begin{gather*}   u_n\wto u= 
 v(x') \exp\left(i\int_0^{x_3} A^{\rm ex}_3(t)\, dt \right)
      \quad\text{in $H^1(\Omega_1;\C)$} \\
A_n- A^{\rm ex}\wto 0 \quad\text{in $\HH(\R^3;\R^3)$} \\
{1\over\eps_n d(x')}\int_{f(x')}^{g(x')} (\partial_3 -i A_{3n})u_n \, dx_3
\wto  b(x') \quad\text{in $L^2(\omega;\C)$}.
\end{gather*}
Here $d(x'):= g(x')-f(x')$, the rescaled thickness of the film.
We observe that the limit $u(x)$ is gauge-equivalent to a function $v(x')$ defined in the 2D domain $\omega$.

The functionals $I_{\kappa,\eps}$ $\Gamma$-converge to the two-dimensional Ginzburg--Landau functional,
$$ I_{\kappa,0}(v,b) = \frac12 \int_{\omega} d(x')\left( \bigl|\bigl(\grad' -i{B^{\rm ex}}' v \bigr|^2 + | b|^2 + \frac{d^2(x')}{12} \bigl|{h^{\rm ex}}'\bigr|^2 |v|^2
	+ \frac{\kappa^2}{2} \bigl( 1 - |v|^2 \bigr)^2  \right) \, dx',
$$
with fixed magnetic vector potential
\begin{equation}\label{Bprime}
 {B^{\rm ex}}':= {h_3^{\rm ex}\over 2}\left(-x_2, x_1\right) - \left({f+g\over 2}\right)\left(-h_2^{\rm ex}, h_1^{\rm ex}\right).  
 \end{equation}
   
%
The quantity $b$ measures the deviation of the gauge-invariant derivative of $u_n$ in the vertical direction, and plays the role of the ``Cosserat vectors'' in limits of elastic membranes (see \cite{BFM},\cite{FFL},\cite{GSM}.)

\medskip

We note two features of the limiting energy.  First, we may recomplete the square in the potential term,
\begin{equation}\label{modpot}   
\frac{d^2(x')}{12} \bigl|{h^{\rm ex}}'\bigr|^2 |v|^2
	+ \frac{\kappa^2}{2} \bigl( 1 - |v|^2 \bigr)^2  
		 = \frac{\kappa^2}{2} \left(
	      \left[1- {d^2(x')|{h^{\rm ex}}'|^2\over 12\kappa^2}\right] - |v|^2
	      \right)^2 +\left[ 
	         \left(1-{d^2(x')|{h^{\rm ex}}'|^2\over 12\kappa^2}\right)^2-1
	         \right].
\end{equation}
Thus, the presence of a strong (order $\rho_\eps\sim\eps^{-1}$) parallel applied field reduces the density of superconducting electrons in the sample, even in the absence of a perpendicular applied field component.  Assume for simplicity that the sample has uniform thickness,
$d(x')=1$.  Then, a simple application of the maximum principle shows that any solution of the Euler--Lagrange equations corresponding to the energy $I_{\kappa,0}$ must satisfy
$$
|v| \leq \argmin_{\rho >0} \frac{1}{12} |{h^{\rm ex}}'|^2 \rho^2 + \frac{\kappa^2}{2} \bigl(1-\rho^2\bigl)^2
	= \argmin_{\rho >0} \left[ \rho^2 - \left(1-\frac{|{h^{\rm ex}}'|^2}{12\kappa^2}\right)  \right]^2
	= \begin{cases}
		0	& \text{ if } |{h^{\rm ex}}'|^2 \geq 12 \kappa^2\\
		\sqrt{1-\frac{|{h^{\rm ex}}'|^2}{12\kappa^2}} & \text{ if } |{h^{\rm ex}}'|^2 < 12 \kappa^2
	\end{cases}
$$
In particular, we conclude that the normal state $v\equiv 0$ is the only solution to the Euler--Lagrange equations for $I_{\kappa,0}$ with
${h^{\rm ex}}'\geq\sqrt{12}\kappa$, that is 
${{\mathbf h}^{\rm ex}}' \gtrsim{\sqrt{12}\kappa\over\eps}$ in the original coordinates.

The second curious consequence in the critical case is the effect of the potential
${B^{\rm ex}}'$.  For films which are appropriately bent (so that $\nabla' (f+g)\neq 0$), the deflection of the film's vertical center essentially converts the horizontal component of the applied field to the vertical, creating a spatially dependent effective field.
Thus, even in the absence of a perpendicular applied field component ($h_3^{\rm ex}=0$) we may observe vortices in the thin film limit, which are approximately vertical, since $v=v(x')$.
For very special domain shapes and applied field strengths, we may even observe vortex concentration on curves in the limit $\kappa\to\infty$, as has been studied by Alama, Bronsard, \& Millot \cite{ABMi}.  We present some illustrative examples in section~\ref{examplesec}.  The proof of the compactness and  $\Gamma$--convergence results will be presented in section~\ref{critsec}.

We note that a similar phenomenon, whereby inhomogeneities in a thin domain lead to a curious dependence on the direction of an applied field, has been observed by Richardson and Rubinstein \cite{RubRich} and proved by Shieh \cite{Shieh} in the context of thin three-dimensional domains which shrink as $\eps\to 0$ to closed space curves.  Shieh also considers $\Gamma$-limits with applied fields on the order of $\eps^{-1}$.
The limiting functional is supported on a closed loop, and it contains a new potential term determined by all three components of the applied field and the geometry of the underlying curve.

\subsection*{The subcritical regime}

The subcritical regime, $\eps\rho_\eps\to 0$, subdivides in two cases.
When $\rho_{\eps} \to \rho < \infty$, we obtain $\Gamma$--convergence results along the lines of the model derived in \cite{CDG}. 
In this case, the magnetic field converges (weakly) to $(0,0,h_3^{\rm ex})$, and through a ``Cosserat vector'' $c =(c_1,c_2)$, we recover the deviation of the parallel magnetic field, $h \approx (\eps c_1, \eps c_2, h^{\rm ex}_3)$. We note that these vectors depend on all three spatial variables, they retain some of the effect of the actual thickness of the film on the deviation of the magnetic field from the vertical, inside and nearby the sample.
The resulting $\Gamma$--limit is the two dimensional Ginzburg--Landau functional
$$ I_{\kappa,-}^{\rho}(u,b,c) = \frac12 \int_{\omega} d(x')\left( \bigl|\bigl(\grad' -i{A^{\rm ex}_{\perp}}') u \bigr|^2 + | b|^2 	+ \frac{\kappa^2}{2} \bigl( 1 - |u|^2 \bigr)^2  \right) \, dx' + \frac12 \int_{\R^3} \bigl|(c_1,c_2) - \rho {h^{\rm ex}}'\bigr|^2 \, dx,
$$
with fixed magnetic potential $A^{\rm ex}_{\perp}=h_3^{\rm ex}(-{x_2\over 2},{x_1\over 2},0)$.

In the case when $\rho_\eps\to \infty$, the magnetic field also converges (weakly) to $(0,0,h^{\rm ex}_3)$, but its parallel deviation is of higher order: $h \approx (\eps\rho_{\eps} c_1, \eps\rho_{\eps} c_2, h^{\rm ex}_3)$, but it doesn't contribute to the energy. In this case, the functionals $I_{\kappa,\eps}$ $\Gamma$--converge to the Ginzburg--Landau functional
$$
I_{\kappa,-}^{\infty}(u,b) = \frac12 \int_{\omega} d(x')\left( \bigl|\bigl(\grad' -i{A^{\rm ex}_{\perp}}') u \bigr|^2 + | b|^2 	+ \frac{\kappa^2}{2} \bigl( 1 - |u|^2 \bigr)^2  \right) \, dx'.
$$
Notice that when the external magnetic field is only applied parallel to the limiting plane ($h_3^{ex}=0$) we recover the simple functional of Bethuel, Brezis, \& H\'elein \cite{BBH}, but with natural (Neumann) boundary conditions.  A precise statement of the compactness and convergence results is in section~\ref{subsec}.  

The case $\rho_{\eps} \to \rho < \infty$ leads to an interesting auxilliary question about divergence-free vector fields:  given the first two components 
$v'=(v_1,v_2)\in L^2(\R^3;\R^2)$ of a vector field on $\R^3$, can it always be completed as a divergence-free vector field $v\in L^2(\R^3;\R^3)$?  It turns out that the answer is no, and we provide an example of a smooth compactly supported $v'$ which may not be completed to a divergence-free $L^2$ vector field.  Fortunately, to construct our upper bounds in the subcritical regime we do not require such a strong result:  it suffices that $v'$ be obtained as a weak limit of divergence-free $L^2$ vector fields, while allowing some unboundedness in the third component.  In section~\ref{SS4.3} we show that any $v'=(v_1,v_2)\in L^2(\R^3;\R^2)$ may be obtained in this way.

\subsection*{The supercritical regime}
In the supercritical regime, $\eps\rho_\eps\to\infty$, the $\Gamma$--limit is trivial:
$$ \Gamma\text{-}\lim_{\eps\to 0} I_{\kappa,\eps}(u_\eps,A_\eps)= \begin{cases}
	\ds \frac{\kappa^2}{4} |\Omega|	& \text{ if } u\equiv 0 \text{ and } h = {h^{\rm ex}}'\\
	\infty		& \text{ otherwise.}
	\end{cases}
$$
This is consistent with the critical case, as taking $\eps\rho_\eps\to L\gg 1$ is
equivalent to multiplying  ${h^{\rm ex}}'$ by a factor $L$ in the previous paragraph.  As described above, when the parallel component of the field is too strong (compared with $\eps^{-1}$) only the normal state is admissible.
A complete analysis of this case will be done in section~\ref{supersec}.

\section{Minimizers of the limit energies}
\label{examplesec}

Before providing the details of the $\Gamma$-convergence results, we discuss some interesting, and in some cases, surprising, consequences for global minimizers of the thin-film limits of Ginzburg--Landau.  The two-dimensional Ginzburg--Landau model has been extensively studied, in particular in the so-called ``London limit'' $\kappa\to \infty$, and here we present some relevant examples and indicate where the pertinent results may be found in the literature.

First we observe that in this section the domains and functions are two-dimensional, and so we use the usual notation $\grad = (\partial_1,\partial_2)$, $x=(x_1,x_2)$. The only exception is the applied magnetic field $h^{\rm ex}$ which is three-dimensional, but the energies yield effective magnetic fields that are vertical, although they may depend on the parallel part of $h^{\rm ex}$.

\medskip

Energy minimizers will (in the $\Gamma$--limit) minimize a two-dimensional functional of the type
\begin{equation}\label{reduced}
G_{\kappa,\lambda}(v; A_0)= \int_\omega d(x)\left\{  
 \frac12 |(\nabla -i\lambda A_0)v|^2 + {\kappa^2\over 4} (|v|^2 - \gamma_\kappa^2)^2\right\} dx.
\end{equation}
In the subcritical case, we may take $\lambda=h_3^{\rm ex}$ and 
$A_0=\frac12 (-x_2,x_1)$.   For the critical case
there are three free parameters, so to reduce their number we fix the direction of the vector field $h^{\rm ex}$ as follows,
\begin{equation}\label{h3d}
    h^{\rm ex}=(h_1^{\rm ex}, h_2^{\rm ex}, h_3^{\rm ex}) 
      = \lambda\left(\alpha_1,\alpha_2,\alpha_3\right),  
\end{equation}
for a constant unit vector $\alpha=(\alpha_1,\alpha_2,\alpha_3)$, $|\alpha|=1$.
In the critical case we thus write 
\begin{equation}\label{critA0}
 A_0= \lambda^{-1}B= (\alpha_2,-\alpha_1)\left[{f+g\over 2}\right]
         + {\alpha_3\over 2}(-x_2,x_1).  
\end{equation}
We note that the only true unknown is $v\in H^1(\Omega;\mathbb{C})$.  The vector potential $A_0$ is given, and write $G_{\kappa,\lambda}(v; A_0)$ to emphasize the dependence of the functional on $A_0$.

The constant $\gamma_\kappa=1$ in the subcritical cases, and is given by
$$  \gamma_\kappa^2= 1- {d^2(x)|h^{\rm ex}{}'|^2\over 12 \kappa^2 }$$
in the critical case.  We will assume that the magnitude of $|h^{\rm ex}{}'|^2\ll \kappa^2$ (and $\kappa\gg 1$) in the following discussion, and so we may effectively think of $\gamma_\kappa=1$ in all cases.

\medskip

We specialize to the case of applied fields on the order of the lower critical field, the value at which vortices first appear in the minimizing configurations.  As is well-known (see \cite{SS},) this occurs at magnetic field strength of order $\lambda\sim\ln\kappa$.  
In this section, we briefly indicate the characteristics of minimizers with vortices in the London limit $\kappa\to\infty$ for general cases and for some interesting examples.  We do not provide proofs, but refer the reader to previous work which applies with few modifications.

\medskip

Assume first that $\omega$ is simply connected;  multiply connected domains require different treatment (see \cite{AB1,AB2}.)  First, we note that this problem exhibits gauge invariance:  for any (smooth) scalar function $\eta$,
there holds
$$   G_{\kappa,\lambda}(v; A_0) = G_{\kappa,\lambda} (v e^{i\eta}; A_0 +\nabla\eta).
$$
In particular, the behavior of minimizers of $G_{\kappa,\lambda}$ will be the same for any vector field $\tilde A_0=A_0+\nabla \eta$ with the same magnetic field $h_0:=\curl A_0$.  It is convenient to exploit the freedom to choose a particular vector potential $A_0$ by {\em fixing a gauge}.  We assume that $A_0$ is chosen such that:
$$  \div (d(x)A_0)=0\quad\text{in $\omega$}, \quad A_0\cdot\nu|_{\partial\omega}=0.  $$
This is always possible, as is proven in \cite{DD1}:  one replaces $A_0$ by $A_0+\nabla'\eta$, and obtains a Neumann problem for $\eta$.  
By this gauge choice, it is possible to find $\xi_0\in H^1_0(\omega)$ with 
$$\nabla^\perp\xi_0=d(x)A_0,  $$
 where $\grad^{\perp} = (-\partial_2, \partial_1)$.  Indeed, $\xi_0$ will solve the Dirichlet problem,
\begin{equation}\label{xi}
\div\left( {1\over d(x)} \nabla \xi_0\right)=\grad^{\perp} \cdot A_0 \ (=h_0), \qquad \xi_0\in H^1_0(\omega).
\end{equation}

It is this auxilliary function $\xi_0$ which will determine the location of the first vortices.  To give an idea of what happens near the first critical field, we present here a formal argument based on a rough evaluation of the energy of vortex configurations.  Assume $v=v_\kappa$ is a minimizer of $G_{\kappa,\lambda}(v; A_0)$ with a finite collection of vortices at points $a_1,\dots,a_m\in\omega$, with degrees $n_1,\dots, n_m$, and the field strength $\lambda\ll\kappa^2$.  For simplicity, take $\gamma_\kappa=1$. We expect that each vortex entails an energy cost, concentrated in a small disk $B=B_{r_i}(a_i)$ centered at the vortex, of the order
$$  \frac12\int_{B_{r_i}(a_i)} d(x)\left\{ |\nabla v|^2 + {\kappa^2\over 2}
(1-|v|^2)^2\right\} \gtrsim \pi |n_i| d(a_i) \ln\kappa.  $$
This energy cost is made precise by the vortex-ball construction in Chapter 4 of \cite{SS}.  The vortices also represent singularities in the Jacobian associated to the map $v$; indeed, for $\kappa$ large,
\begin{equation}\label{jac}
  Jv_\kappa=\det Dv_\kappa=\frac12 \nabla\times (iv_\kappa, \nabla v_\kappa)
   \simeq  \pi \sum_{j=1}^m  n_j \delta_{a_j}.  
\end{equation}
This may be made explicit using the work of Jerrard \& Soner \cite{JS}, and the above approximation holds in the norm on the dual space to $C_0^{0,1}(\overline{\omega})$.  
To see why vortices are produced, at which field strength $\lambda$, and at which points in $\omega$, we expand the energy of minimizers $v_\kappa$:
\begin{align}\nnn
G_{\kappa,\lambda}(v_\kappa; A_0) &= \frac12\int_\omega 
   d(x) \left\{
         |\nabla v_\kappa|^2 
          -2 \lambda A_0\cdot (i v_\kappa,\nabla v_\kappa)
            + \lambda^2 |A_0|^2 |v_\kappa|^2 
              + {\kappa^2\over 2}(|v_\kappa|^2-1)^2\right\} \\
              \nnn
             &\gtrsim    
             \pi\sum_{i=1}^m |n_i| d(a_i) \ln\kappa  - 
               \lambda \int_\omega \nabla^\perp\xi_0 
                   \cdot (iv_\kappa,\nabla v_\kappa) 
                     + {\lambda^2\over 2}\int_\omega d(x) |A_0|^2 \\
   \label{LB1}
             &\simeq
               \pi\sum_{i=1}^m |n_i| d(a_i) \ln\kappa
                + 2\pi\lambda \sum_{i=1}^m n_i \xi_0(a_i) 
                 + {\lambda^2\over 2}\int_\omega d(x) |A_0|^2,
\end{align}
where we have integrated by parts and used \eqref{jac} in the last line.
A simple upper bound on the energy of minimizers is obtained using $v\equiv 1$ as a test function,
$$   G_{\kappa,\lambda}(v_\kappa; A_0)\le  
    G_{\kappa,\lambda}(1;A_0) =
      {\lambda^2\over 2}\int_\omega d(x) |A_0|^2.
$$
  In order to have vortices, the cost of each vortex (estimated by the first term on the right-hand side in \eqref{LB1}) should be balanced by the second term,
  that is,
$$   \sum_{i=1}^m \left\{ |n_i| d(a_i) \ln\kappa + 2\lambda n_i \xi_0(a_i)
                           \right\} \lesssim 0.  $$
For the second term to be as large (negative) as possible, we should place vortices at or near the point set 
$$  \Lambda:=\left\{p\in\omega:  \  \left| {\xi_0(p)\over d(p)}\right|
       = \max_\omega \left| {\xi_0\over d}\right|\right\}.  $$
at which the maximum of $|\xi_0/d|$ is attained.  If the value of $\xi_0/d$ is positive there, the degree $n_i<0$, while if the value of $\xi_0/d$ is negative, the vortex should have degree $n_i>0$.  The critical value of $\lambda$ at which the two terms are exactly balanced gives the lower critical field, which is given by
$$   H_{c1}= {1\over 2\max_\omega |\xi_0(x)/d(x)|} \ln\kappa + O(1).  $$
That is, for $\lambda=\lambda(\kappa) < H_{c1}$, there should be no vortices, since they cost more energy than they save, while for larger $\lambda$ energy minimization favors the creation of vortices near the set $\Lambda$.
These computations are formal, but may be made precise using the methods of \cite{SS}.

\medskip

\noindent {\bf The Subcritical Case.} \ 
In the subcritical case, $A_0={A^{\rm ex}_{\perp}}'$, corresponds to the constant vertical field $h_3^{\rm ex}$, and $\gamma_\kappa=1$.
Numerical simulations of this model have been undertaken in \cite{CDG, LDu}, and in the case of simply-connected domains $\omega$, a study of global minimizers with vortices has been undertaken by Ding \& Du \cite{DD1, DD2}, in the limit $\kappa\to\infty$.  In this setting, $\grad^{\perp}\cdot A_0\equiv 1$, so
by the maximum principle, $\xi_0<0$ in $\omega$.  Assuming $d(x)$ is real-analytic, $\xi_0/d$ attains its global minimum at a finite number of points interior to $\omega$.  In this case, the result of \cite{DD2} applied directly, and for applied fields sufficiently close to $H_{c1}$,
$\lambda=h_3^{\rm ex}= H_{c1} + K\ln \ln\kappa$,
a finite number of vortices (the number uniformly bounded in $\kappa$) of positive degree will concentrate as $\kappa\to \infty$ near the set of minimizers of $\xi_0/d$.  This outcome is qualitatively identical to the corresponding result for the usual two-dimensional Ginzburg--Landau model, and so the thin film geometry does not play a special role in the subcritical case for applied fields close to the critical field $H_{c1}$.

\medskip

We note that the hypothesis that $\omega$ be simply-connected is implicit in the arguments of \cite{DD1,DD2}, which no longer hold for multiply-connected domains.  As was observed in \cite{AAB}, in a multiply-connected domain the holes act as ``giant vortices'' at {\em bounded} applied field strength $h^{\rm ex}_3$.  To analyze the creation of vortices in the interior of $\omega$ the effect of the holes must be taken into account, modifying the choice of auxilliary function which determines the critical field and the vortex locations.
This analysis was done for a circular annulus (in the context of Bose-Einstein condensates) in \cite{AAB}, and extended to more general multiply-connected domains and the full Ginzburg--Landau functionals (with or without inhomogeneities) in \cite{AB2, AB1}.  In these papers it has been observed that  vortices may concentrate on {\em curves} in multiply-connected $\omega$ as $\kappa\to\infty$.  The asymptotic distribution of vortices along the limiting curve is studied in \cite{ABMi}.

\medskip

\noindent
{\bf The Critical Case.} \ 
In the critical regime more interesting phenomena may be observed.
As mentioned above, $\gamma_\kappa\sim 1$, and so the reduction of $|v|$ by the modification of the potential \eqref{modpot} is negligible for applied fields $h^{\rm ex}=O(\ln\kappa)$.  However, the effective vector potential 
(see \eqref{critA0}) yields some new, unexpected results for the London limit
$\kappa\to\infty$.
Indeed, the equation for $\xi_0$ now reads as:
\begin{equation}\label{xi1}
   -\div\left({1\over d(x)}\nabla \xi_0\right) 
     = -\nabla^{\perp}\cdot A_0=\alpha\cdot\left( {\partial\over\partial x_1}\left[{f+g\over 2}\right]
     , {\partial\over\partial x_2}\left[{f+g\over 2}\right], -1\right)  \quad
\xi_0\in H^1_0(\omega).  
\end{equation}
Note that the effective magnetic field $\nabla^{\perp}\cdot A_0$ coincides with the projection of the field direction $\alpha$ onto the familiar area-weighted normal vector to the centroid surface
$x_3=\frac12(f(x)+g(x))$.  In particular, we observe that if the film's centroid
is not planar, then the function $\xi_0$ is modified, and thus the lower critical field and location of vortices will differ from 
the subcritical case, due to the presence of the parallel field components ${h^{\rm ex}}'$.

Since the right-hand side of \eqref{xi1} may not be sign definite, we cannot conclude from the Maximum principle that $\xi_0$ is sign definite, leading to the possibility that the maximum of $|\xi_0/d|$ could occur at a positive or negative value of $\xi_0$.  Denote by
$$  \Lambda:=\left\{p\in\omega:  \  \left| {\xi_0(p)\over d(p)}\right|
       = \max_\omega \left| {\xi_0\over d}\right|\right\}.  $$
In case the maxima of $|\xi_0/d|$ occur at finitely many points in $\omega$, an analysis similar to that of \cite{SS, DD2} applies, and we may prove:  

\begin{theorem}\label{pointthm} 
Assume $\Lambda$ consists of finitely many points, and there exist constants $C,M>0$ for which
\begin{equation}\label{SScond}
\left| {\xi_0(x)\over d(x)}\right| \leq \max_\omega \left| {\xi_0\over d}\right|
   - C[\text{dist}\,(x,\Lambda)]^{M}, 
\end{equation}
for $x$ in some neighborhood of $\Lambda$.
  Let $\alpha\in\R^3$ be a constant unit vector and 
$$  h^{\rm ex}=\alpha\,\lambda(\kappa)= \alpha\left[ 
{1\over 2\max_{\omega}|\xi_0/d|}\ln\kappa + K\ln\ln\kappa\right],  $$
with fixed constant $K$.  For any sequence $\kappa_n\to\infty$,
let $v_n$ be the minimizer of the energy $I_{\kappa_n,0}$, with $A_0$ as in \eqref{critA0}.  Then:
\begin{enumerate}
\item  there exists $K_*\in\R$ so that if $K<K_*$, $v_n$ has no vortices for all large $n$.
\item  for any $K\geq K_*$, $v_n$ has finitely many vortices, and the sum of the absolute values of their degrees is uniformly bounded in terms of $K$.
\item  the vortices concentrate at points in $\Lambda$, in the sense that their distance to $\Lambda$ is bounded by $(\ln\kappa)^{-\beta}$ for  constant $\beta>0$.
\item  if $p\in\Lambda$ and $\xi_0(p)<0$, the vortices concentrating at $p$ have positive degrees.  If $\xi_0(p)>0$, the degrees are negative. 
\end{enumerate}
\end{theorem}
The proof of this result follows that of \cite{SS2}, except it is necessary to treat points of $\Lambda$ in two groups, those with positive and negative values of $\xi_0$.  We note that hypothesis \eqref{SScond} holds when $d(x),f(x),g(x)$ are real-analytic.

We note that in this context,  it is possible (and natural) that the maximum of $|\xi_0/d|$ is attained at both positive and negative values of $\xi_0$, in which case minimizers would exhibit {\em both} vortices and antivortices.  This will be the case if we choose $\omega=D_1(0)$, the unit disk, with $f(x)= \frac12 |x|^2$, $g(x)=f(x)+1$ (and thus $d(x)=1$.)
Then, taking a horizontal field, $\alpha=(1,0,0)$, we may solve the equation for $\xi_0$ exactly, $\xi_0(x)=\frac18 x_1 (1-|x|^2)$.  The
maximum absolute value is attained at $x=(\pm {1\over\sqrt{3}}, 0)$,
giving positive degree vortices concentrating at $(-{1\over\sqrt{3}}, 0)$ and
negative degree (anti-)vortices at $({1\over\sqrt{3}}, 0)$.
Since the thin film limit leads to $v=v(x_1,x_2)$, the vortices are essentially veritical, and thus the infinitesimal curvature of the film thus engenders {\em vertical} vortex lines in response to a purely {\em horizontal} applied field!

\medskip

Furthermore, it is also possible to find settings in which the maximum of $|\xi_0/d|$ is attained on a curve inside $\omega$, either a closed curve or a collection of compactly contained arcs.  For instance, if we again consider the case of a disk $\omega=D_1(0)$, but now choose a different thickness profile $f(x) =\frac{|x| x_2}{2} + \frac{x_1^2}{2}\ln\bigl(\frac{|x|+x_2}{|x_1|}\bigr)$, $g(x)=f(x)+1$ (so again $d(x)=1$), with applied field generated by $\alpha=(1,0,0)$, we may again solve for $\xi_0$ explicitly, obtaining:
$$\xi_0(x)= {1\over 8}r(1-r^2), \quad r = |x|.$$
The maximum value is obtained on the circle $r=|x|=1/\sqrt{3}$.  In this setting, we may apply the following $\Gamma$--convergence theorem of Alama, Bronsard, \& Millot \cite{ABMi}: 
suppose  $d(x)\equiv 1$, and define
$$   J_\kappa (v):= G_{\kappa,\lambda}(v) - \frac12\int_\omega |\lambda A_0|^2.  $$
\begin{theorem}
Assume $d(x)\equiv 1$, $\Lambda$ is a $C^2$ Jordan curve or embedded arc in $\omega$,
$\xi_0<0$ in $\omega$, attaining its minimum on $\Lambda$, and
\eqref{SScond} holds.
Assume 
$$  h^{\rm ex}=\alpha\,\lambda(\kappa)= \alpha\left[ 
{1\over\max_{\omega}|\xi_0|}\ln\kappa + \beta(\kappa)\right],  $$
with $1\ll \beta(\kappa)\ll \ln\kappa$.
  Let $\kappa_n\to\infty$.  Then:
\begin{enumerate}
\item  for any $v_n$ with $\sup_n {J_{\kappa_n}(v_n)\over \beta^2(\kappa_n)}<\infty$, there is a subsequence and a nonnegative Radon measure $\mu\in H^{-1}(\omega)$ supported on $\Lambda$ so that
$$   {1\over\beta(\kappa_n)} \curl (iv_n, \nabla' v_n) \to \mu \qquad
\text{strongly in $(C_0^{0,1}(\omega))^*$.}
$$
\item  The family ${1\over\beta^2(\kappa)}J_{\kappa}$ of functionals $\Gamma$-converges to $J_\infty(\mu)= I(\mu) - \|\xi_0\|_\infty \mu(\omega),$ where
$$  I(\mu) = \frac12 \iint_{\omega\times\omega}  G(x,y)\, d\mu(x)
d\mu(y) ,  $$
and $G$ is the Dirichlet Green's function of the domain $\omega$.
\item
If $v_n$ is a sequence of global minimizers of $J_{\kappa_n}$, then
$$   {1\over\beta(\kappa_n)} \curl (iv_n, \nabla' v_n) 
\to  {\|\xi_0\|_\infty\over 2I_*}\mu_* ,
$$
where $\mu_*$ is the unique probability measure which minimizes $I$, and 
$I_*=I(\mu_*)$.
\end{enumerate}
\end{theorem}

In other words, energy minimizers in this setting will have a large number, $O(\beta(\kappa))$, of point vortices concentrating near the curve $\Lambda$,
and their distribution along $\Lambda$ will be governed by the electrostatic potential $I(\mu)$.  Thus, the distribution of vortices is determined by a classical equilibrium measures problem from potential theory (see \cite{Ran, ST}.)  In the above example, $\Lambda$ is a circle in the disk $\omega=D_1(0)$, and the measure $\mu$ is normalized arclength.  Hence, the vortices will be asymptotically uniformly distributed on the circle.

\begin{figure}[htp]
\centering
\includegraphics[width=2.5in]{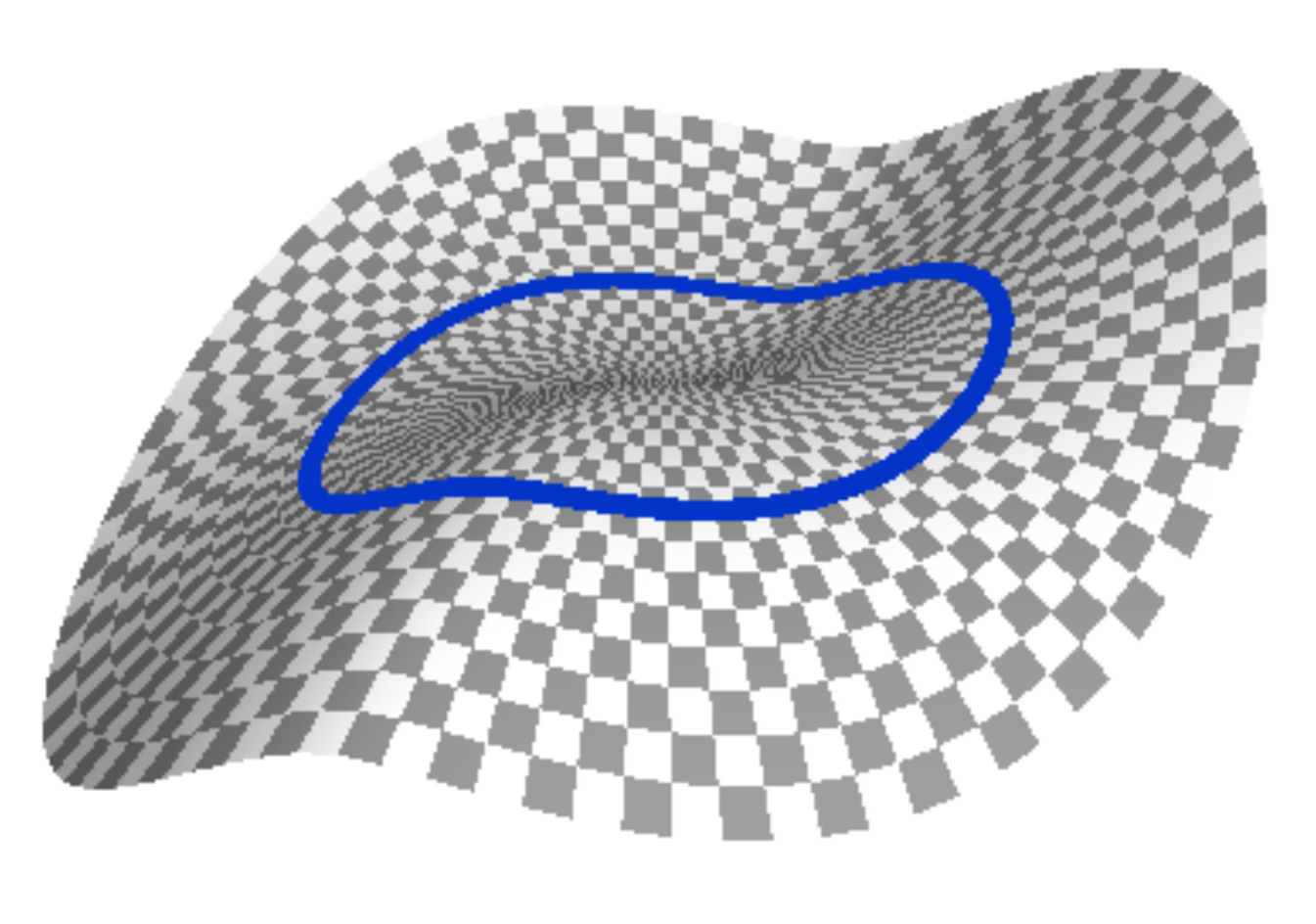}
\caption{The centroid given by $x_3=f(x')=\frac{|x'| x_2}{2} + \frac{x_1^2}{2}\ln\bigl(\frac{|x'|+x_2}{|x_1|}\bigr)$,
with external field directions $\alpha=(\alpha_1,0,0)$.  Near the lower critical field, the vortices concentrate near the circle shown.}\label{fig1}
\end{figure}


\section{Critical Case}
\label{critsec}

We begin proving the $\Gamma$-convergence results, starting with the critical case, $\eps\rho_\eps\to L\in (0,\infty)$.  For simplicity we assume $\rho_\eps=\eps^{-1}$; for other limits $L$ we incorporate the value of $L$ in ${h^{ex}}{}'$.  Following \cite{GP}, we define the Hilbert space
$\HHr$ as the completion of the space $C_0^\infty(\R^3;\R^3)$ of smooth, compactly supported divergence-free vector fields in the Dirichlet norm,  
 $\| F \|_{\HH}=[\int_{\R^n} |DF|^2\, dx]^{1/2}$.  
   It follows that $F\in \HH(\R^3;\R^3)$ is divergence-free in the sense of distributions.  We may not have $F\in L^2(\R^3;\R^3)$ (and so $\HHr\neq H^1_{\rm div}(\R^3;\R^3)$), but by the Sobolev embedding $F\in L^6(\R^3;\R^3)$.  We will require the following useful result on $\HHr$ from \cite{GP}.

\begin{lemma}\label{lem:GP}
\begin{enumerate}
\item
Let $g \in L^2(\R^3;\R^3)$ such that $\div g = 0$ in $\mathcal{D}'(\R^3)$.
Then there is a unique $F \in \HHr$ such that $\grad \times F = g$ and $\div F = 0$.
\item  For any $F\in \HHr$, there exists a constant $C$ with
$$  \| F \|_{L^6} \leq C \| F \|_{\HH} \leq C \| \nabla\times F\|_{L^2}.  $$
\end{enumerate}
\end{lemma}

Here and throughout the paper, we denote by $\nabla^\perp = (-\partial_2, \partial_1, 0)$ and hence for a vector field $F$,
$$   \nabla^\perp\cdot F = \partial_1 F_2 - \partial_2 F_1, $$
a shorthand for the third component of the curl of $F$.

\medskip

With our simplifying assumption $\eps\rho_\eps=1$, we consider vector potentials of the form $A= B+ A^{\rm ex}$, with $B\in \HHr$ and fixed ($\eps$-independent) external potential
$$  A^{\rm ex}= \frac12 h^{\rm ex}\times (x_1,x_2,x_3) =
      \frac12( h^{\rm ex}_2 x_3-h^{\rm ex}_3 x_2, h^{\rm ex}_3 x_1 - h^{\rm ex}_1 x_3, h^{\rm ex}_1 x_2 - h^{\rm ex}_2 x_1).
$$

Since we are only interested in the limit as $\eps \to 0^+$, and we keep $\kappa$ fixed, we drop the subscript $\kappa$ from the functional for simplicity of notation.
For $(u,A-A^{\rm ex}) \in H^1(\Omega;\C) \times \Hdiv(\R^3;\R^3)$, we recall
\begin{multline*}
I_{\eps}(u,A) := 
	 \frac12 \int_{\Omega} \left( | (\grad'-iA')u|^2 + \left| \frac1\eps \bigl( \partial_3 - iA_3\bigr)u \right|^2 + \frac{\kappa^2}{2} \bigl( 1 - |u|^2 \bigr)^2 \right) \, dx  \\
	+ \frac12 \int_{\R^3} \left(|h_3 - h^{\rm ex}_3|^2 + \frac{1}{\eps^2} \left| h' - \eps \rho_{\eps}{h^{\rm ex}}' \right|^2 \right)  \, dx.
\end{multline*}


We now state the complete $\Gamma$-convergence result in three parts:  the compactness of sequences of bounded energy; the lower semicontinuity of the limit; and the existence of sequences $\eps_n$, $(u_n,A_n)$ for which the energies converge.
The appropriate limiting space is:
\begin{equation}\label{eq:V_crit}
\mathcal{V}_0:= H^{1}(\omega; \C) \times L^2(\omega;\C),
\end{equation}

%


\begin{theorem}\label{thm:critical}
Let $\eps_n \to 0^+$ as $n\to+\infty$. Then
\begin{enumerate}
\item[(i)] for any sequence $\{(u_{n},A_n-A^{\rm ex})\} \in H^{1}(\Omega;\C) \times \HHr$ such that
$$
\sup_{n \in \NN} I_{\eps_n}(u_n,A_n) < \infty,
$$
there exists a subsequence (not relabeled) $\{(u_n,A_n-A^{\rm ex})\}$ and $(v,b) \in \mathcal{V}_0$ such that
\begin{align}
& u_n \wto u:=ve^{iA^{\rm ex}_3x_3}  \text{ in } H^{1}(\Omega;\C), \label{eq:crit1}\\
& A_n-  A^{\rm ex} \wto 0 \text{ in } \HHr, \\ 
& v_n := u_ne^{-i\int_0^{x_3} (A_n)_3(x',t)\,dt} \wto v \text{ in } H^{1}(\Omega;\C), \\
& b_n := \frac{1}{d(x')\eps_n} \int_{f(x')}^{g(x')} \partial_3 v_n(x',t) \, dt \wto b \text{ in }L^2(\omega;\C), \label{eq:critlast}
\end{align}
where  $d(x') := g(x')- f(x')$.

\item[(ii)] for any sequence $\{(u_{n},A_n-A^{\rm ex})\} \in H^{1}(\Omega;\C) \times \HHr$ satisfying \eqref{eq:crit1}--\eqref{eq:critlast} for some $(v,b) \in \mathcal{V}_0$, the $\Gamma$-limit of $I_{\eps_n}(u_n,A_n)$ is
$$
I_{k,0}(v,b) := \frac12 \int_{\omega} d(x') \left( \bigl|\bigl(\grad' -i{B^{\rm ex}}'\bigr)v \bigr|^2 + | b|^2 + \frac{d^2(x')}{12} \bigl|{h^{\rm ex}}'\bigr|^2 |v|^2
	+ \frac{\kappa^2}{2} \bigl( 1 - |v|^2 \bigr)^2  \right) \, dx
$$
\end{enumerate}
\end{theorem}

\begin{proof}[ of Theorem \ref{thm:critical} (i)]

Let $K := \sup_{n \in \NN} I_{\eps_n}(v_n, b_n, A_n) < \infty$.
Then
\begin{equation}\label{unifbd}
\int_{\R^3} \left(\frac{1}{\eps_n^2}\left| h_n' - {h^{\rm ex}}'\right|^2 + |h_3 - h^{\rm ex}_3|^2\right) \, dx \leq K.
\end{equation}
In particular, Lemma~\ref{lem:GP} implies that
$$
\sup_{n\in\NN} \|\grad (A_n - A^{\rm ex})\|_{L^2(\R^3;\R^{3\times 3})} < \infty.
$$
Thus,  we deduce that $\{A_n - A^{\rm ex}\}$ is bounded in $\HH$ and in $L^6(\R^3;\R^3)$, thus there exists a subsequence (not relabeled) such that
$$
A_n- A^{\rm ex} \wto A \quad \text{ in } \HHr.
$$
By weak convergence we have $\div A=0$, and by the uniform bound
\eqref{unifbd} we may conclude that $\|h'_n- h^{ex}{}'\|_{L^2}\to 0$, and thus
$$
\grad \times A = 0.
$$

Hence by the uniqueness in Lemma \ref{lem:GP}, we deduce that $A = 0$.

Moreover, we know that $\{u_n\}_{n \in \NN}$ is bounded in $L^4 (\Omega;\C)$,
and because $\grad A_n$ is bounded in $L^2(\R^3;\R^{3\times 3})$, 
\begin{align*}
& u_n \text{ is bounded in } L^2(\Omega;\C), \\
& \grad u_n \text{ is bounded in } L^2(\Omega;\C^3),
\end{align*}
so there exists a further subsequence (not relabeled) such that
$$
u_n \wto u \quad \text{ in } H^1(\Omega;\C).
$$

Also, if we define $v_n(x) = u_n(x) e^{-i\int_{0}^{x_3} {(A_n)}_3(x',t)\,dt}$, then we have
\begin{align*}
& |v_n|=|u_n| \text{, which is bounded in } L^2(\Omega;\C), \\
& \grad v_n \text{ is bounded in } L^2(\Omega;\C^3), \\
& \partial_3 v_n \to 0 \quad \text{ in } L^2(\Omega;\C),
\end{align*}
so we deduce that there is a further subsequence (not relabeled) such that
$$
v_n \wto v \quad \text{ in } H^1(\Omega;\C)
$$
with $\partial_3 v = 0$.
We then have that 
\begin{equation}\label{eq:uv}
u = v e^{i \int_{0}^{x_3} A^{\rm ex}_3(x',t)dt} = v e^{i A^{\rm ex}_3 x_3}.
\end{equation}

On the other hand, we know that $b_n$ is bounded in $L^2(\omega;\C)$, hence there exists a subsequence (not relabeled) and a function $b \in L^2(\omega;\C)$ such that
$$
b_n \wto b \quad \text{ in } L^2(\omega;\C).
$$

This completes the proof of part (i).
\end{proof}


To prove part (ii) of Theorem~\ref{thm:critical}, we derive matching upper and lower bounds.  We begin with:

\begin{proposition}[$\Gamma-\liminf$ inequality]
Let $(v,b) \in \mathcal{V}_0$ and consider sequences $\{\eps_n\} \subset \R$, $\{u_n\} \subset H^1(\Omega;\C)$, and $\{A_n-A^{\rm ex}\} \subset \HHr$ satisfying 
\begin{align*}
& \eps_n \to 0^+, \\
& u_n \wto u:=v e^{i A^{\rm ex}_3 x_3} \text{ in } H^1(\Omega;\C), \\
& v_n := u_n e^{-i \int_0^{x_3} {(A_n)}_3(x',t)dt} \wto v \text{ in } H^1(\Omega;\C), \\
& b_n:= \frac{1}{\eps_n d(x')} \int_{f(x')}^{g(x')} \partial_3 v_n \, dx_3 \wto b \text{ in } L^2(\omega;\C), \\
& A_n - A^{\rm ex} \wto  0 \text{ in } \HHr. 
\end{align*}

Then
$$
\liminf_{n \to \infty} I_{\eps_n}(u_n, A_n) 
	\geq \frac12 \int_{\omega} d(x') \left( \bigl|\bigl(\grad' -i{B^{\rm ex}}'\bigr) v\bigr|^2 + |b|^2 + \frac{d^2(x')}{12} |{h^{\rm ex}}'|^2 |v|^2
	+ \frac{\kappa^2}{2} \bigl( 1 - |v|^2 \bigr)^2  \right) \, dx',
$$
where $B^{\rm ex}$ is as on Theorem \ref{thm:critical}.
\end{proposition}

\begin{proof}[]

Since $v_n \to v$ in $H^1(\Omega;\C)$ and $|u_n|=|v_n|$, we know that $v_n \to v$ in $L^2(\Omega;\C) \cap L^6(\Omega;\C)$, hence
$$
\lim_{n \to \infty} \int_{\Omega} \frac{\kappa^2}{4} \bigl( 1 - |u_n|^2 \bigr)^2 \, dx 
	= \lim_{n \to \infty} \int_{\Omega} \frac{\kappa^2}{4} \bigl( 1 - |v_n|^2 \bigr)^2 \, dx 
	= \int_{\Omega} \frac{\kappa^2}{4} \bigl( 1 - |v|^2 \bigr)^2 \, dx.
	= \int_{\omega} d(x') \frac{\kappa^2}{4} \bigl( 1 - |v|^2 \bigr)^2 \, dx'.
$$

Then because $b_n \wto b$ in $L^2(\omega;\C)$ and
\begin{align*}
\liminf_{n \to \infty} \frac12 \int_{\Omega}\left| \frac{1}{\eps_n}(\partial_3 -i{(A_n)}_3)u_n \right|^2 \, dx 
	& \geq \liminf_{n \to \infty} \frac12 \int_{\omega}\left| \int_{f}^{g} \frac{1}{\eps_n}(\partial_3 -i{(A_n)}_3)u_n \, dx_3 \right|^2 \, dx'  \\
	& = \liminf_{n \to \infty} \frac12 \int_{\omega} d(x') |b_n|^2 \, dx' 
		\geq \frac12 \int_{\omega} d(x') |b|^2 \, dx',
\end{align*}
using Fubini's theorem, H\"older's inequality, and Fatou's lemma. 

Moreover, 
$$
\grad \times A_n - h^{\rm ex}\wto 0 \text{ in } L^2(\R^3;\R^3),
$$ 
so we write
$$
\bigl|(\grad' - iA_n') u_n \bigr|^2 = |\grad' u_n|^2 + |A_n' u_n|^2 + 2\,\Re(A_n' \ol{u}_n \cdot \grad' u_n).
$$

Using the fact that $\grad' u_n \wto \grad' u$ in $L^2(\Omega;\C)$, 
$$
\liminf_{n \to \infty} \int_{\Omega} |\grad' u_n|^2 \,dx 
	\geq \int_{\Omega} |\grad' u|^2 \,dx,
$$
and since $A_n \to A^{\rm ex}$ in $H^1_{loc}$ and $u_n \to u$ in $L^4(\Omega;\C)$, 
$$
\lim_{n \to \infty} \int_{\Omega} |A'_n u_n|^2 \,dx 
	= \int_{\Omega} |{A^{\rm ex}}' u|^2 \,dx.
$$

Also $A'_n \ol{u}_n \to {A^{\rm ex}}' \ol{u}$ in $L^2(\Omega;\C)$, so 
$$
\lim_{n \to \infty} \int_{\Omega} 2\,\Re\bigl(A'_n \ol{u}_n \cdot \grad' u_n\bigr) \, dx
	= \int_{\Omega} 2\,\Re\bigl({A^{\rm ex}}' \ol{u} \cdot \grad' u\bigr) \, dx.
$$

This yields
$$
\liminf_{n \to \infty} \int_{\Omega} \bigl|(\grad' - iA'_n) u_n \bigr|^2 \, dx
	\geq \int_{\Omega} \bigl|\bigl(\grad' - i{A^{\rm ex}}'\bigr) u \bigr|^2 \, dx.
$$

To complete the proof, we write the last term in a different form:
\begin{align}
\bigl|\bigl(\grad' - i(A^{\rm ex})'\bigr) u \bigr|^2
	& = \bigl|\bigl(\grad' - i\bigl({A^{\rm ex}_{\|}}' + {A^{\rm ex}_{\perp}}' \bigr)\bigr) u \bigr|^2 \nonumber\\
	& = \bigl|\bigl(\grad' - i\bigl({A^{\rm ex}_{\|}}' + {A^{\rm ex}_{\perp}}' + {\ts \frac12}(h_2^{\rm ex},-h^{\rm ex}_1) x_3\bigr) v\bigr|^2 \nonumber\\
	& = \bigl|\bigl(\grad' - i\bigl( {A^{\rm ex}_{\perp}}' + (h_2^{\rm ex},-h^{\rm ex}_1) x_3 \bigr)\bigr) v \bigr|^2 \nonumber\\
	& = \bigl|\bigl(\grad' - i {A^{\rm ex}_{\perp}}'\bigr) v \bigr|^2  + \frac12|{h^{\rm ex}}'|^2 |v|^2 x_3^2 
		+ 2 \Im\bigl( \bigl(\grad' - i {A^{\rm ex}_{\perp}}'\bigr) v \cdot (h_2^{\rm ex},-h^{\rm ex}_1) \ol{v} \bigr) x_3, \label{eq:indepx3}
\end{align}
where we recall that $A^{\rm ex}_{\|} = \frac12 (h^{\rm ex}_2 x_3, -h^{\rm ex}_1 x_3, h^{\rm ex}_1 x_2 - h^{\rm ex}_2 x_1)$, so $\grad \times A^{\rm ex}_{\|} = (h^{\rm ex}_1,h^{\rm ex}_2,0)$, and we recall that $A^{\rm ex}_{\perp} = \frac{h^{\rm ex}_3}{2} (-x_2,x_1,0)$.
Since we know that none of the terms in \eqref{eq:indepx3} depends on $x_3$, and
$$
\int_{f}^{g} x_3 \, dx_3 = \frac{g^2-f^2}{2} = \frac{d (f+g)}{2}
\qquad \text{ and } \qquad
\int_{f}^{g} x_3^2 \, dx_3 = \frac{g^3-f^3}{3} = \frac{d(f^2+fg+g^2)}{3},
$$
we deduce that
\begin{align}
\int_{\Omega} \bigl|(\grad' - i{A^{\rm ex}}') u \bigr|^2 \, dx
	& = \int_{\omega} \int_{-\frac12}^{\frac12} \bigl|(\grad' - i{A^{\rm ex}}') u \bigr|^2 \, dx_3 \, dx' \notag\\
	& = \int_{\omega} d(x') \biggl( \bigl|\bigl(\grad' - i {A^{\rm ex}_{\perp}}'\bigr) v \bigr|^2 + 2 \Im\bigl( \bigl(\grad' - i {A^{\rm ex}_{\perp}}'\bigr) v \cdot (h_2^{\rm ex},-h^{\rm ex}_1) \ol{v} \bigr) \left(\frac{f+g}{2}\right) \notag\\
	& \qquad + |{h^{\rm ex}}'|^2 |v|^2 \left(\frac{f^2+fg+g^2}{3}\right) \biggr) \, dx'  \notag\\
     & = \int_{\omega} d(x') \biggl( \bigl|\bigl(\grad' - i {B^{\rm ex}}'\bigr) v  \bigr|^2  + |{h^{\rm ex}}'|^2 |v|^2 \left(\frac{f^2+fg+g^2}{3} - \frac{(f+g)^2}{4}\right) \biggr) \, dx'  \notag\\
     & = \int_{\omega} d(x') \biggl( \bigl|\bigl(\grad' - i {B^{\rm ex}}'\bigr) v  \bigr|^2  + \frac{d^2(x')}{12}|{h^{\rm ex}}'|^2 |v|^2 \biggr) \, dx',  \label{eq:Bv} 
\end{align}
where ${B^{\rm ex}}' = {A^{\rm ex}_{\perp}}' + {\ts \frac{f+g}{2}}(-h^{\rm ex}_2,h^{\rm ex}_1)$ and on the third equality we completed the square.
This completes the proof.
\end{proof}






We complete the proof of Theorem~\ref{thm:critical} by means of an upper bound construction:

\begin{proposition}[$\Gamma-\limsup$ inequality]
\label{ub:critical}

Let $(v,b) \in \mathcal{V}_0$ and let $\{\eps_n\} \in \R$ be a sequence such that $\eps_n \to 0^+$.
Then, there exist sequences $\{v_n\} \subset H^1(\Omega;\C)$, $\{b_n\} \subset L^2(\omega;\C)$, and $\{A_n-A^{\rm ex}\} \subset \HHr$ such that
\begin{align*}
& \eps_n \to 0^+, \\
& u_n := v_n e^{i \int_0^{x_3} {(A_n)}_3(x',t)dt} \wto u:=v e^{i A^{\rm ex}_3 x_3} \text{ in } H^1(\Omega;\C), \\
& v_n \wto v \text{ in } H^1(\Omega;\C), \\
& b_n = \frac{1}{\eps_n d(x')} \int_{f(x')}^{g(x')} \partial_3 v_n \,dx_3\wto b \text{ in } L^2(\omega;\C), \\
& A_n - A^{\rm ex} \wto 0\text{ in } \HHr. 
\end{align*}
and
$$
\lim_{n \to \infty} I_{\eps_n}(u_n,A_n) 
	=\frac12 \int_{\omega} d(x') \left( \bigl|\bigl(\grad' -i{B^{\rm ex}}'\bigr) v\bigr|^2 + | b|^2 + \frac{d^2(x')}{12} |{h^{\rm ex}}'|^2 |v|^2
	+ \frac{\kappa^2}{2} \bigl( 1 - |v|^2 \bigr)^2  \right) \, dx',
$$
where $B^{\rm ex}$ is as on Theorem \ref{thm:critical}.
\end{proposition}

\begin{proof}[]

Define
$$
A_n(x) := A^{\rm ex}(x), 
$$
and
\begin{equation}\label{eq:def_un}
u_n(x) 
	= e^{i A^{\rm ex}_3(x') x_3} \bigl( v(x')+\eps_n b(x') x_3 \bigr).
\end{equation}


We prove first that the convergences in the proposition hold.
Note that
$$ 
|u_n - u| 
	= \left | \eps_n e^{i A^{\rm ex}_3(x')x_3} b(x')x_3 \right|
	= \eps_n  \left| b(x') x_3\right|
$$ 
so that
$$
\|u_n - u\|_{L^2(\Omega;\C)}^2 
	\leq \eps_n \|b\|_{L^2(\omega;\C)}^2 \to 0.
$$

Since $\{\grad u_n\}$ is bounded in $L^2(\Omega;\C^3)$, we know that
$$
u_n \wto u \text{ in } H^1(\Omega;\C).
$$

The other convergences are trivial, since $b_n \equiv b$ and $A_n \equiv A^{\rm ex}$.
Moreover,
$$
I_{\eps_n}(u_n,A_n)
	= \frac12 \int_{\Omega} \left( | (\grad'-iA_n)u_n|^2 + |b|^2 + \frac{\kappa^2}{2} \bigl( 1 - |u_n|^2 \bigr)^2 \right) \, dx.
$$

By expanding the last term above using \eqref{eq:def_un}, we have that
$$
\lim_{n\to \infty} \frac12 \int_{\Omega} \bigl(1-|u_n|^2\bigr)^2 \, dx 
	=\frac12 \int_{\omega} d(x') \bigl(1-|v|^2\bigr)^2 \, dx'.
$$

As for the remaining term, we follow an analogous reasoning as in \eqref{eq:Bv} to deduce
$$
\int_{\Omega} \bigl| (\grad'-iA_n')u_n\bigr|^2 \, dx
	 = \int_{\omega} d(x') \left( \bigl| \bigl(\grad'-i{B^{\rm ex}}'\bigr)v\bigr|^2 + \frac{d^2(x')}{12} |{h^{\rm ex}}'|^2 |v|^2 \right) \, dx' +O(\eps_n).
$$

We conclude that
$$
\lim_{n \to \infty} I_{\eps_n}(u_n,A_n)
	= \frac12 \int_{\omega} d(x') \left( \bigr| \bigl(\grad'-i{B^{\rm ex}}'\bigr)v\bigr|^2 + \left| b \right|^2+ \frac{d^2(x')}{12} |{h^{\rm ex}}'|^2 |v|^2 +  \frac{\kappa^2}{2} \bigl( 1 - |v|^2 \bigr)^2 \right) \, dx'.
$$
This completes the proof of Proposition~\ref{ub:critical}, and with it Theorem~\ref{thm:critical}.
\end{proof}


\section{Subcritical Case}
\label{subsec}

This case, when $\eps \rho_{\eps} \to 0$, is itself split into two subcases, when $\rho_{\eps} \to \rho\in [0,\infty)$ and $\rho_{\eps} \to \infty$.
We recall the definition of $A^{\rm ex}_\eps$ from \eqref{Aex}; note that 
in this regime $A^{\rm ex}_\eps\to A^{\rm ex}_\perp$ (in $H^1_{loc}$) with limiting potential
$$  A^{\rm ex}_\perp := \frac12 h^{ex}_3 \vec{e}_3\times (x_1,x_2,x_3) =
 h^3_{ex}  \left( -{x_2\over 2}, {x_1\over 2},0\right).  $$

To capture the Cosserat vectors in the limit we must have some control on the order of $\eps$ at which the vector fields are converging or diverging.  
We thus define the space
\begin{equation}\label{eq:V_subcrit_0}
\mathcal{V}_-:=  H^{1}(\omega; \C) \times L^2(\omega;\C) \times L^2(\R^3;\R^2).
\end{equation}

We consider sequences $\eps_n\to 0$, and write $\rho_n=\rho_{\eps_n}$ 
and $A^{\rm ex}_n=A^{\rm ex}_{\eps_n}$ throughout.

\begin{theorem}\label{thm:subcrit_fin}
Let $\eps_n \to 0^+$ and $\rho_n \to \rho \in [0, \infty)$ be arbitrary sequences. Also set $\ds  c'_n: =\frac{1}{\eps_n}(\grad \times A_n)'$.
Then
\begin{enumerate}
\item[(i)] for any sequence $\{(u_n, A_n-A^{ex}_n)\} \subset H^1(\Omega;\C) \times \Hdiv(\R^3;\R^3)$ such that 
$$
\sup_{n\in\NN} I_{\eps_n}(u_n,A_n) < \infty,
$$
there exists a subsequence (not relabeled) $\{(u_n, A_n-A^{ex}_n)\}$ and $(u,b,c' - \rho {h^{\rm ex}}') \in \mathcal{V}_-$ such that
\begin{align}
& u_n \wto u  \text{ in } H^{1}(\Omega;\C), \label{eq:sub_comp_1}\\
& A_n-A^{\rm ex}_n \wto 0  \text{ in } \HHr, \\
& b_n  \wto b \text{ in }L^2(\omega;\C), \\
& c_n'-\rho_n h^{ex}{}' \wto  c' - \rho h^{ex}{}' \text{ in }L^2(\R^3;\R^2).\label{eq:sub_comp_2}
\end{align}

\item[(ii)] for any sequence $\{(u_n, A_n-A^{ex}_n)\} \subset H^1(\Omega;\C) \times L^2(\omega;\C) \times \Hdiv(\R^3;\R^3)$ satisfying \eqref{eq:sub_comp_1}--\eqref{eq:sub_comp_2} for some $(u,b,c' - \rho {h^{\rm ex}}') \in \mathcal{V}_-$, the $\Gamma$--limit of $I_{\eps_n}(u_n,A_n)$ is
$$
I_{\kappa,-}^{\rho}(u,b,c') := \frac12 \int_{\omega} d(x') \left( \bigl|\bigl(\grad' -i{A^{\rm ex}_{\perp}}'\bigr) u\bigr|^2 + |b|^2
	+ \frac{\kappa^2}{2} \bigl( 1 - |u|^2 \bigr)^2  \right) \, dx' + \frac12 \int_{\R^3} |c' - \rho {h^{\rm ex}}'|^2 \, dx.
$$
\end{enumerate}
\end{theorem}

\begin{theorem}\label{thm:subcrit_inf}
Let $\eps_n \to 0^+$, $\limsup \rho_n = \infty$, and set $\ds  c'_n: =\frac{1}{\eps_n\rho_n}(\grad \times A_n)'$.
Then
\begin{enumerate}
\item[(i)] for any sequence $\{(u_n, A_n-A^{ex}_n)\} \subset H^1(\Omega;\C) \times \Hdiv(\R^3;\R^3)$ such that 
$$
\sup_{n\in\NN} I_{\eps_n}(u_n,A_n) < \infty,
$$
there exists a subsequence (not relabeled) $\{(u_n, A_n-A^{ex}_n)\}$ and $(u,b) \in \mathcal{V}_0$ such that
\begin{align}
& u_n \wto u  \text{ in } H^{1}(\Omega;\C), \label{eq:sub2_comp_1}\\
& A_n-A^{\rm ex}_n \wto 0  \text{ in } \HHr, \\
& b_n  \wto b \text{ in }L^2(\omega;\C), \\
& c_n' -  h^{ex}{}' \wto  0 \text{ in }L^2(\R^3;\R^2).\label{eq:sub2_comp_2}
\end{align}

\item[(ii)] for any sequence $\{(u_n, A_n-A^{ex}_n)\} \subset H^1(\Omega;\C) \times \Hdiv(\R^3;\R^3)$ satisfying \eqref{eq:sub2_comp_1}--\eqref{eq:sub2_comp_2} for some $(u,b) \in \mathcal{\nu}_0$, the $\Gamma$--limit of $I_{\eps_n}(u_n,A_n)$ is
$$
I_{\kappa,-}^{\infty}(u,b,c') := \frac12 \int_{\omega} d(x') \left( \bigl|\bigl(\grad' -i{A^{\rm ex}_{\perp}}'\bigr) u\bigr|^2 + |b|^2
	+ \frac{\kappa^2}{2} \bigl( 1 - |u|^2 \bigr)^2  \right) \, dx'.
$$
\end{enumerate}
\end{theorem}

\begin{corollary}

Theorems \ref{thm:subcrit_fin} and \ref{thm:subcrit_inf}, imply that the Ginzburg-Landau model in $3D$
$$
\min_{\substack{u \in H^1(\Omega;\C) \\ A \in H^1_{\div}(\R^3;\R^3)}} I_{\eps}(u,A)
$$
converges, in the thin-film limit, to the model
$$
\min_{u \in H^1(\omega;\C)} \frac12 \int_{\omega} d(x')\left( \bigl|\bigl(\grad' -i{A^{\rm ex}_{\perp}}'\bigr)u\bigr|^2 + \frac{\kappa^2}{2} \bigl( 1 - |u|^2 \bigr)^2 \right)  \, dx',
$$
where we let $b\equiv 0$ in $\omega$ and $c'\equiv {h^{\rm ex}}'$ in $\R^3$.
\end{corollary}

\subsection{Compactness}

\begin{proof}[ of Theorem \ref{thm:subcrit_fin} {\it (i)}]

Let $K := \sup_{n \in \NN} I_{\eps_n}(u_n,A_n) < \infty$.
Then
\begin{equation}\label{enest}
\int_{\R^3} \frac{1}{\eps_n^2} \left| h_n' - \eps_n \rho_n{h^{\rm ex}}'\right|^2 + |h_3 - h^{\rm ex}_3|^2 \, dx \leq K,
\end{equation}
This implies that $\grad \times (A_n-  A^{\rm ex}_n)$ is bounded in $L^2$, and by Lemma~\ref{lem:GP} we conclude that $(A_n-  A^{\rm ex}_n)$ is 
bounded in $\HHr$, and therefore there exists  a subsequence (not relabeled) such that
$$
B_n:= (A_n-  A^{\rm ex}_n) \wto B \quad \text{ in } \HHr.
$$
Then, by weak convergence, $\div B=0$, and by the estimate \eqref{enest} we conclude that 
$
\grad \times B = (0,0,\grad^{\perp}\cdot B).
$
This implies that $\partial_3[ \grad^{\perp}\cdot B] = 0$ in $\mathcal{D}'(\R^3)$. Also, from Fatou's Lemma in \eqref{enest}, we deduce that $\grad^{\perp}\cdot B \in L^2(\R^3)$, thus $\grad^{\perp} \cdot B \equiv 0$.
The uniqueness in Lemma \ref{lem:GP} implies that $B \equiv 0$.

This means that in the thin film limit, the magnetic field is vertical. The Cosserat vector for the magnetic field should give the direction which the magnetic field takes to get vertical in the limit. \\

Since $\rho_n \to \rho \in [0,\infty)$, we have that
%
$$
\int_{\R^3} \left| c_n' - \rho_n{h^{\rm ex}}'\right|^2 \, dx = \int_{\R^3} \left| \frac{1}{\eps_n} h_n' - \rho_n{h^{\rm ex}}'\right|^2 \, dx \leq K,
$$
which implies that we can find a further subsequence (not relabeled) such that
$$
c_n' -\rho_n h^{ex}{}' \rightharpoonup c' -\rho h^{ex}{}'
\quad \text{ in } L^2(\R^3;\R^2).
$$
%


%


On the other hand,
$\{u_n\}_{n \in \NN}$ is bounded in $L^4 (\Omega;\C)$,
and because $\grad A_n$ is bounded in $L^2(\R^3;\R^{3\times 3})$, 
\begin{align*}
& u_n \text{, which is bounded in } L^2(\Omega;\C), \\
& \grad u_n \text{ is bounded in } L^2(\Omega;\C^3),
\end{align*}
so there exists a further subsequence (not relabeled) such that
$$
u_n \wto u \quad \text{ in } H^1(\Omega;\C).
$$

Also, if we define $v_n(x) = u_n(x) e^{-i\int_0^{x_3} {(A_n)}_3(x',t)\,dt}$, then we have
\begin{align*}
& |v_n|=|u_n| \text{, which is bounded in } L^2(\Omega;\C), \\
& \grad v_n \text{ is bounded in } L^2(\Omega;\C^3), \\
& \partial_3 v_n \to 0 \quad \text{ in } L^2(\R^3;\C),
\end{align*}
so we deduce that there is a further subsequence (not relabeled) such that
$$
v_n \wto v \quad \text{ in } H^1(\Omega;\C)
$$
with $\partial_3 v = 0$.
We then have that 
\begin{equation}\label{eq:uv2}
u = v e^{i \int_0^{x_3} {A^{\rm ex}_{\perp}}_3(x',t)dt}= v.
\end{equation}

Recall that $b_n := \frac{1}{\eps_n d(x')}\int_{f(x')}^{g(x')} \partial_3 v_n \,dx_3\in L^2(\omega;\C)$.
Then, we know that $b_n$ is bounded in $L^2(\omega;\C)$, hence there exists a subsequence (not relabeled) and a function $b \in L^2(\omega;\C)$ such that
$$
b_n \wto b \quad \text{ in } L^2(\omega;\C).
$$
\end{proof}

\begin{proof}[ of Theorem \ref{thm:subcrit_inf} {\it (i)}]

The proof of the compactness result for the case when $\limsup \rho_n = \infty$ follows the same proof as in the previous case, so given $K := \sup_{n \in \NN} I_{\eps_n}(u_n,A_n) < \infty$, we can find a subsequence (not relabeled) such that
$$
A_n - A^{\rm ex}_n \wto 0
$$

Since $\limsup \rho_n = \infty$, we know that 
$$
\rho_n^2 \int_{\R^3} \left| c_n - {h^{\rm ex}}'\right|^2 \, dx = \rho_n^2 \int_{\R^3} \left| \frac{1}{\eps_n\rho_n} h_n' - {h^{\rm ex}}'\right|^2 \, dx \leq K,
$$
which implies that we can find a further subsequence (not relabeled) such that
$$
c_n-{h^{\rm ex}}' \rightharpoonup 0 \quad \text{ in } L^2(\R^3;\R^2).
$$

From here, we follow the previous proof without change to obtain a further subsequence (not relabeled) such that
\begin{align*}
& u_n \wto u \quad \text{ in } H^1(\Omega;\C), \\
& v_n \wto u \quad \text{ in } H^1(\Omega;\C), \\
& b_n \wto b \quad \text{ in } L^2(\omega;\C).
\end{align*}
\end{proof}

\subsection{The $\Gamma$-liminf inequality}

\begin{proposition}[$\Gamma-\liminf$ inequality]

\begin{enumerate}
\item[(i)] Let $\bigl(u,b,c' - \rho{h^{\rm ex}}'\bigr) \in \mathcal{V}_-$ and consider sequences $\{\eps_n\} \subset \R$, $u_{n}\in H^{1}(\Omega;\C)$ and $A_n-A^{\rm ex}_n \in \HHr$, $n\in\NN$,  satisfying
\begin{align*}
& \eps_n \to 0^+, \quad \rho_n \to \rho ,  \\
& u_n \wto u \text{ in } H^1(\Omega;\C), \\
& b_n \wto b \text{ in } L^2(\omega;\C), \\
& A_n-A^{\rm ex}_{n} \wto 0 \text{ in } \HHr,  \\
& c_n'-\rho_n h^{ex}{}'  \wto c' -\rho h^{ex}{}'\text{ in }L^2(\R^3;\R^2),
\end{align*}
with $c'_n$ as in Theorem \ref{thm:subcrit_fin}. 
Then
$$
\liminf_{n \to \infty} I_{\eps_n}(u_n,A_n) 
	\geq \frac12 \int_{\omega} d(x') \left( \bigl|\bigl(\grad' -i{A^{\rm ex}_{\perp}}'\bigr) u\bigr|^2 + |b|^2
	+ \frac{\kappa^2}{2} \bigl( 1 - |u|^2 \bigr)^2  \right) \, dx' 
		+ \frac12 \int_{\R^3} |c' - \rho {h^{\rm ex}}'|^2  \, dx.
$$

\item[(ii)] Let $(u,b) \in \mathcal{V}_0$ and consider sequences $\{\eps_n\} \subset \R$, $u_{n}\in H^{1}(\Omega;\C)$ and $A_n-A^{\rm ex}_n \in \HHr$, $n\in\NN$, satisfying
\begin{align*}
& \eps_n \to 0^+, \quad \limsup \rho_n = \infty,  \\
& u_n \wto u \text{ in } H^1(\Omega;\C), \\
& b_n \wto b \text{ in } L^2(\omega;\C), \\
& A_n -A^{\rm ex}_{n}\wto 0 \text{ in } \HHr,  \\
& c_n' -h^{ex}{}'\wto 0 \text{ in }L^2(\R^3;\R^2),
\end{align*}
with $c'_n$ as in \ref{thm:subcrit_inf}.
Then
$$
\liminf_{n \to \infty} I_{\eps_n}(u_n,A_n) 
	\geq \frac12 \int_{\omega} d(x')\left( \bigl|\bigl(\grad' -i{A^{\rm ex}_{\perp}}'\bigr) u\bigr|^2 + |b|^2
	+ \frac{\kappa^2}{2} \bigl( 1 - |u|^2 \bigr)^2  \right) \, dx.
$$

\end{enumerate}
\end{proposition}

\begin{proof}[]

Since $u_n \to u$ in $H^1(\Omega;\C)$, we know that $u_n \to u$ in $L^2(\Omega;\C) \cap L^6(\Omega;\C)$, hence
$$
\lim_{n \to \infty} \int_{\Omega} \frac{\kappa^2}{4} \bigl( 1 - |u_n|^2 \bigr)^2 \, dx 
	= \int_{\Omega} \frac{\kappa^2}{4} \bigl( 1 - |u|^2 \bigr)^2 \, dx
	= \int_{\omega} d(x')\frac{\kappa^2}{4} \bigl( 1 - |v|^2 \bigr)^2 \, dx'.
$$

Then because $b_n(x) \wto b$ in $L^2(\Omega;\C)$ and
$$
\liminf_{n \to \infty} \frac12 \int_{\Omega}\left| \frac{1}{\eps}(\partial_3 -i{(A_n)}_3)u_n \right|^2 \, dx 
	\geq \liminf_{n \to \infty} \frac12 \int_{\omega} d(x')|b_n|^2 \, dx 
		\geq \frac12 \int_{\omega} d(x') |b|^2 \, dx',
$$
using Fubini's theorem, H\"older's inequality, and Fatou's lemma. 

For the covariant term, 
we write
$$
\bigl|(\grad' - iA_n') u_n \bigr|^2 = |\grad' u_n|^2 + |A_n' u_n|^2 + 2\,\Re(A_n' \ol{u}_n \cdot \grad' u_n).
$$

Using the fact that $\grad' u_n \wto \grad' u$ in $L^2(\Omega;\C)$, 
$$
\liminf_{n \to \infty} \int_{\Omega} |\grad' u_n|^2 \,dx 
	\geq \int_{\Omega} |\grad' u|^2 \,dx,
$$
and since $A_n \to A^{\rm ex}_{\perp}$ in $H^1_{loc}$, and $u_n \to u$ in $L^4(\Omega;\C)$, 
$$
\lim_{n \to \infty} \int_{\Omega} |A'_n u_n|^2 \,dx 
	= \int_{\Omega} \bigl|{A^{\rm ex}_{\perp}}' u\bigr|^2 \,dx.
$$

Also $A'_n \ol{u}_n \to {A^{\rm ex}_{\perp}}' \ol{u}$ in $L^2(\Omega;\C)$, so 
$$
\lim_{n \to \infty} \int_{\Omega} 2\,\Re(A'_n \ol{u}_n \cdot \grad' u_n) \, dx
	= \int_{\Omega} 2\,\Re\bigl({A^{\rm ex}_{\perp}}' \ol{u} \cdot \grad' u\bigr) \, dx.
$$

This yields
$$
\liminf_{n \to \infty} \int_{\Omega} \bigl|(\grad' - iA'_n) u_n \bigr|^2 \, dx
	\geq \int_{\Omega} \bigl|\bigl(\grad' - i{A^{\rm ex}_{\perp}}'\bigr) u \bigr|^2 \, dx'.
	= \int_{\omega} d(x')\bigl|\bigl(\grad' - i{A^{\rm ex}_{\perp}}'\bigr) u \bigr|^2 \, dx'.
$$

Finally, in case  {\it (i)} we apply Fatou's Lemma to the last term,
$$
\liminf_{n \to \infty} \frac12 \int_{\R^3} \left| \frac{1}{\eps_n} (\grad \times A_n)' - \rho_n {h^{\rm ex}}' \right|^2\, dx
	\geq \frac12 \int_{\R^3}  \left|c' - \rho {h^{\rm ex}}' \right|^2 \, dx.
$$
This completes the proof.
\end{proof}

\subsection{The $\Gamma$ limsup inequality}\label{SS4.3}

As mentioned in the Introduction, the Cosserat vectors in the case
$\rho_n\to\rho$ are the rescaled limit of the $x'$-component of the internal magnetic field.  More specifically, by the compactness result, Theorem~\ref{thm:subcrit_fin}~(i), in case $\rho_n\to\rho\ge 0$,
$$   w'_n:= {1\over\eps_n}(\nabla\times A_n)' -\rho_n h^{ex}{}'
   \wto c' - \rho h^{ex}{}'=: w',  $$
and $w'\in L^2(\R^3;\R^2)$.  In order to construct upper bound sequences we need to recover sequences $w_n\in L^2_{div}(\R^3;\R^3)$ whose first two components converge to $w'$.  

As a first attempt, we may ask whether a given $w'$ may be extended to
$w=(w',w_3)$, a divergence-free $L^2(\R^3;\R^3)$ vector field.  It turns out that this is not possible, even for smooth compactly supported $w'\in C_c^\infty(\R^3;\R^2)$.  Consider the following example:
let $\varphi(x)\in C_c^\infty(\R^3)$ with 
$$   \begin{cases} 
\varphi(x)=1, &\text{for  } {\ds \max_{j\in\{1,2,3\}} |x_j|\leq 1}, \\
\varphi(x)=0, &\text{for } {\ds \max_{j\in\{1,2,3\}} |x_j|\geq 2}, \\
\varphi(x)\geq 0, &\text{in $\R^3$,}
\end{cases} $$
and $w'(x)=(x_1,x_2)\varphi(x)$.  Assume that we can find $w_3(x)$ so that
$w=(w_1,w_2,w_3)\in L^2(\R^3;\R^3)$ with divergence zero.  In that case,
we calculate $\partial_{x_3} w_3 = -2\varphi + (x_1,x_2)\cdot\nabla'\varphi$.  For $(x_1,x_2)=(0,0)$ we conclude
$\partial_{x_3} w_3 =-2\varphi\le 0$ for all $x_3\in\R$, and 
$\partial_{x_3}w_3 =-2$ for $x_3\in [-1,1]$.  In particular, $w_3(0,0,x_3)$ has distinct limits as $x_3\to\pm\infty$, and thus $w\not\in L^2(\R^3;\R^3)$.

Fortunately, we do not require $w'\in L^2(\R^3;R^2)$ to be the restriction of a divergence-free $L^2$ vector field, and we may make indeed recover any
$w'\in L^2(\R^3;\R^2)$ as a limit of divergence-free vector fields as in Theorem~\ref{thm:subcrit_fin}.





\begin{lemma}\label{lem:cosserat}
Let $w' \in L^2(\R^3;\R^2)$. 
Then there is a sequence $\{B_{\eps}\}_{\eps>0} \subset \HHr$ such that $(\grad \times B_{\eps})'\to w'$ in $L^2(\R^3;\R^2)$ and $\eps \grad \times B_{\eps}\to 0$ in $L^2(\R^3;\R^3)$.
\end{lemma}

\begin{proof}[]
We divide the proof in three steps.

\paragraph{Step 1.} $w'$ is the characteristic function of a compact set.

Assume that $w'(x) = (1,1)\chi_{K}(x)$ where $K \subset \R^3$ is a compact set. Then, for all $\delta>0$, define $w'_{\delta} := w' * \psi_{\delta}$ where $\psi_{\delta}(x) = \frac{1}{\delta^3}\psi\bigl(\frac{x}{\delta}\bigr)$ and $\psi \in C^{\infty}_c(\R^3)$ is the standard mollifier.
Consider ${(w_{\delta})}_3(x) := -\int_0^{x_3} \bigl( \partial_1 {(w_{\delta})}_1(x',t) + \partial_2 {(w_{\delta})}2(x',t) \bigr) \,dt \in C^{\infty}(\R^3;\R^2)$, and $\div w_{\delta} = 0$.

Consider the function $\chi_{\eta} \in C^{\infty}(\R)$ such that $\chi_{\eta}(t) \equiv 1$ for $|t| \leq \eta$ and $\chi_{\eta}(t) \leq C\exp\bigl(-{(t^2-\eta^2)}^2\bigr)$, $\chi_{\eta} \geq 0$, and $\|\chi_{\eta}'\|_{\infty} \leq C$. Now we define $W_{\eta,\delta}(x) := w_{\delta}(x) \chi_{\eta}(x_3) - \grad \varphi_{\eta,\delta}(x)$, where $\varphi_{\eta,\delta} \in H^1(\R^3)$ is the solution of $\Delta \varphi_{\eta,\delta}(x) = {(w_{\delta})}_3(x) \chi'_{\eta}(x_3)$.

Since $\div W_{\eta,\delta} = 0$, by Lemma \ref{lem:GP}, we find $B_{\eta,\delta} \in \HHr$ such that $\grad \times B_{\eta,\delta} = W_{\eta,\delta}$.

On the other hand, since $w_{\delta} \in L^{\infty}(\R^3;\R^3)$, we have that $\Delta \varphi_{\eta,\delta} \to 0$ as $\eta \to \infty$ in $L^2(\R^3)$, and 
$$
\int_{\R^3} |\grad \varphi_{\eta,\delta}|^2 \, dx 
	= - \int_{\R^3} {(w_{\delta})}_3 \chi_{\eta}'(x_3) \varphi_{\eta,\delta} \, dx
	\leq \| {(w_{\delta})}_3 \chi_{\eta}'\|_{L^{\frac65}(\R^3)} \|\varphi_{\eta,\delta}\|_{L^6(\R^3)}
	\to 0 \text{ as } \eta \to \infty.
$$
Thus $\varphi_{\eta,\delta} \to 0$ as $\eta \to \infty$ in $H^1(\R^3)$, which means that $(\grad \times B_{\eta,\delta})' \to w'_{\delta}$ as $\eta \to \infty$ in $L^2(\R^3;\R^2)$. Then we can find $\eta_{\delta} \to \infty$ as $\delta \to 0^+$ such that 
$$
\|(\grad \times B_{\eta_{\delta},\delta})' - w'_{\delta}\|_{L^2(\R^3;\R^2)} \leq \delta.
$$
Denote $B_{\delta} := B_{\eta_{\delta},\delta}$ and $W_{\delta} := W_{\eta_{\delta},\delta}$. Then,
\begin{align*}
\int_{\R^3} \left| \grad \times B_{\delta}' - w'\right|^2 \, dx
	& \leq  2 \int_{\R^3} \left|\grad \times B_{\delta}' - w'_{\delta}\right|^2 \, dx  + 2 \int_{\R^3} |w'_{\delta} - w'|^2 \, dx \\
	& \leq 2\delta^2 + 2 \|w'_{\delta} - w'\|_{L^2(\R^3;\R^2)}^2 
	\to 0 \text{ as } \delta \to 0^+,
\end{align*}
so $(\grad \times B_{\delta})' \to w'$ in $L^2(\R^3;\R^2)$, which implies that $\eps(\grad \times B_{\delta_{\eps}})' \to 0$ in $L^2(\R^3;\R^2)$ for all $\delta_{\eps} \to 0^+$.

For the third component of the curl, we may choose $\delta_{\eps} \to 0^+$ as $\eps \to 0^+$ such that 
$$
\|(W_{\delta_{\eps}})_3\|_{L^2(\R^3;\R^3)} \leq \frac{1}{\sqrt{\eps}}.
$$
This yields $\eps W_{\delta_{\eps}} \to 0$ in $L^2(\R^3;\R^3)$.

\paragraph{Step 2.} $w'$ is a simple function with compact support.

Since these functions are just a finite sum of characteristic functions of compact sets, the proof follows immediately from Step 1.

\paragraph{Step 3.} General case.

Let $w' \in L^2(\R^3;\R^2)$.
Then, we can find a sequence of simple functions with compact support $\{w'_n\}$ such that $w'_n \to w'$ in $L^2(\R^3;\R^2)$.

Then, following the construction in Step 1, we can find a sequence $B_{n,\delta} \in \HHr$ satisfying
$$
(\grad \times B_{n,\delta})'\to w'_n \text{ as } \delta \to 0^+ \text{ in } L^2(\R^3;\R^2).
$$
Hence we can find $\delta_n \to 0^+$ such that
$$
\|(\grad \times B_{n,\delta_n})' - w'_n\|_{L^2(\R^3;\R^2)} \leq \frac1n,
$$
thus $(\grad \times B_{n,\delta_n})'\to w'  \text{ in } L^2(\R^3;\R^2)$. We write $B_n := B_{n,\delta_n}$.
Then $\eps(\grad \times B_{n_{\eps}})' \to 0$ in $L^2(\R^3;\R^2)$ for all $n_{\eps} \to \infty$.

For the third component of the curl, we make use of the extra $\eps$ by choosing $n_{\eps} \to \infty$ as $\eps \to 0^+$ such that
$$
\|(W_{n_{\eps}})_3\|_{L^2(\R^3;\R^3)} \leq \frac{1}{\sqrt{\eps}}.
$$
This yields $\eps W_{n_{\eps}} \to 0$ in $L^2(\R^3;\R^3)$. Write $B_{\eps} := B_{n_{\eps}}$ and the proof is complete.
\end{proof}

\begin{proposition}[$\Gamma-\limsup$ inequality for Theorem~\ref{thm:subcrit_fin}]

Let $\bigl(u,b,c' - \rho{h^{\rm ex}}'\bigr) \in \mathcal{V}_-$ and let $\{\eps_n\} \in \R$ be a sequence such that $\eps_n \to 0^+$ and $\rho_n \to \rho$.
Then, there exist sequences $\{u_n\} \subset H^1(\Omega;\C)$ and $\{A_n-A^{\rm ex}_n\} \subset \HHr$ such that
\begin{align*}
& u_n \wto u \text{ in } H^1(\Omega;\C), \\
& b_n \equiv b \text{ a.e. in } \omega, \\
& A_n-A^{\rm ex}_{n} \to 0 \text{ in } \HHr,  \\
& c_n'-\rho_n h^{ex}{}'  \to c' -\rho h^{ex}{}'\text{ in } L^2(\R^3;\R^2),
\end{align*}
with $c'_n$ as in Theorem \ref{thm:subcrit_fin},
and
$$
\lim_{n \to \infty} I_{\eps_n}(u_n,A_n) 
	= \frac12 \int_{\omega} \left( \bigl|\bigl(\grad' -i{A^{\rm ex}_{\perp}}'\bigr) u\bigr|^2 + | b|^2
	+ \frac{\kappa^2}{2} \bigl( 1 - |u|^2 \bigr)^2  \right) \, dx' 
		+ \frac12 \int_{\R^3} \left| c - \rho {h^{\rm ex}}'\right|^2\, dx.
$$
\end{proposition}

\begin{proof}[]
Applying Lemma \ref{lem:cosserat} to $w'=c'-\rho h^{ex}{}'$, we find a sequence of potentials $B_n \in \HHr$.  We define
$$
A_n(x) := A^{\rm ex}_{n}(x) + \eps_n B_n(x),
$$
so that
$$ c'_n = {1\over\eps_n} (\nabla\times A_n)' = 
       \rho_n h^{ex}{}' + (\nabla\times B_n)', $$
and $(\nabla\times B_n)'\to w'=c'-\rho h^{ex}{}'$.
Then define
$$
u_n(x) 
	= e^{i \eps_n \int_0^{x_3} (B_n)_3(x',t)\,dt} \bigl( u(x')+\eps_n b(x')x_3 \bigr).
$$
Then, we prove first that the convergences in the proposition hold.

First, note that
\begin{align*}
|u_n - u| 
	& = \left | u \left( e^{i \eps_n \int_0^{x_3} (B_n)_3(x',t)\,dt} -1\right) + \eps_n e^{i \eps_n  \int_0^{x_3} (B_n)_3(x',t)\,dt} b(x')x_3\right| \\
	& \leq |u| \left| e^{i \eps_n \int_0^{x_3} (B_n)_3(x',t)\,dt} -1\right| + \eps_n  \left| b(x')\right|
\end{align*}
so that
\begin{align*}
& |u_n - u|  \to 0 \text{ a.e.  in } \Omega, \\
& |u_n-u|^2 \leq 4 |u| + C \eps_n \text{ which is integrable in } \Omega.
\end{align*}

Using Lebesgue's Dominated Convergence, we deduce that
$$
u_n \to u \text{ in } L^2(\Omega;\C).
$$

Since $\{\grad u_n\}$ is bounded in $L^2(\Omega;\C^3)$, we know that
$$
u_n \wto u \text{ in } H^1(\Omega;\C),
$$
and $b_n \equiv b$.

Also, $\eps_n B_n \to 0$ in $\HHr$, so 
$$
A_n - A^{\rm ex}_{n} \wto 0 \text{ in } \HHr.
$$

By convergence of $B'_n$, we have that
\begin{align*}
\lim_{n\to\infty} {1\over\eps_n^2} \int_{\R^3} 
|h'-\eps_n\rho_n h^{ex}{}'|^2
&=
\lim_{n\to\infty} \int_{\R^3} |c_n' - \rho_n {h^{\rm ex}}'|^2 \, dx \\
& = \int_{\R^3} |c' - \rho{h^{\rm ex}}'|^2 \, dx,
\end{align*}
and because $\eps_n \grad \times B_n \to 0$ in $L^2(\R^3;\R^3)$, we have that
$$
\lim_{n\to\infty} \int_{\R^3} |(\nabla\times A_n)_3 - h^{\rm ex}_3|^2 \, dx = 0.
$$

Moreover, we know that
$$
|u_n(x)|^2 
	= \left| u(x') + \eps_n b(x')x_3\right|^2
	= |u|^2 + O_{L^1}(\eps_n) 
$$
thus
$$
\bigl(1-|u_n|^2\bigr)^2 
	= \bigl(1-|u|^2\bigr)^2 + \, O_{L^1}(\eps_n) 
	\to \bigl(1-|u|^2\bigr)^2 
	\text{ in } L^1(\Omega).
$$

By Lebesgue's Dominated Convergence Theorem, we obtain that
$$
\lim_{n\to \infty} \frac12 \int_{\Omega} \bigl(1-|u_n|^2\bigr)^2 \, dx 
	= \frac12 \int_{\omega} \bigl(1-|u|^2\bigr)^2 \, dx'.
$$

As for the covariant term, we have
\begin{align*}
\bigl| (\grad'-iA_n')u_n\bigr|^2
	 = \bigl| (\grad'-i{A^{\rm ex}_{\perp}}')u\bigr|^2  + O_{L^1}(\eps_n),
\end{align*}
hence
$$
\lim_{n\to\infty} \frac12 \int_{\Omega} \bigl| (\grad'-iA_n')u_n\bigr|^2 \, dx 
	= \frac12 \int_{\omega} \bigl| \bigl(\grad'-i{A^{\rm ex}_{\perp}}'\bigr)u\bigr|^2 \, dx'.
$$

This completes the proof.
\end{proof}

\begin{proposition}[$\Gamma-\limsup$ inequality for Theorem~\ref{thm:subcrit_inf}]

Let $(u,b) \in \mathcal{V}_0$ and let $\{\eps_n\} \in \R$ be a sequence such that $\eps_n \to 0^+$ and $\limsup \rho_n = \infty$.
Then, there exist sequences $\{u_n\} \subset H^1(\Omega;\C)$ and $\{A_n-A^{\rm ex}_n\} \subset \HHr$  such that
\begin{align*}
& u_n \wto u \text{ in } H^1(\Omega;\C), \\
& b_n \equiv b \text{ a.e. in } \omega, \\
& A_n -A^{\rm ex}_{n}\to 0 \text{ in } \HHr,  \\
& c_n' \equiv h^{ex}{}' \text{ a.e. in } \R^3,
\end{align*}
with $c'_n$ as in Theorem \eqref{thm:subcrit_inf},
and
$$
\lim_{n \to \infty} I_{\eps_n}(u_n,A_n) 
	= \frac12 \int_{\omega} \left( \bigl|\bigl(\grad' -i{A^{\rm ex}_{\perp}}'\bigr) u\bigr|^2 + | b|^2
	+ \frac{\kappa^2}{2} \bigl( 1 - |u|^2 \bigr)^2  \right) \, dx'.
$$
\end{proposition}

\begin{proof}[]
Define
$$
A_n(x) := A^{\rm ex}_{\perp}(x') + \eps_n \rho_n A^{\rm ex}_{\|}(x),
$$
and
$$
u_n(x) 
	:= e^{i \eps_n\rho_n \int_0^{x_3} {A^{\rm ex}_{\|}}_3(x',t)\,dt} \bigl( u(x')+\eps_n b(x')x_3 \bigr).
$$

Then, we prove first that the convergences in the proposition hold.
As in the previous proof, we deduce that
$$
u_n \wto u \text{ in } H^1(\Omega;\C),
$$
and $b_n \equiv b$.
Also, 
$$
A_n - A^{\rm ex}_{\perp}\to 0 \text{ in } \HHr,
$$
and
$$
c_n' \equiv {h^{\rm ex}}' \quad \text{ in } L^2(\R^3;\R^2).
$$

Moreover,
$$
I_{\eps_n}(u_n,A_n)
	= \frac12 \int_{\Omega} \left( | (\grad'-iA_n')u_n|^2 + |b|^2 + \frac{\kappa^2}{2} \bigl( 1 - |u_n|^2 \bigr)^2 \right) \, dx.
$$

Following the same reasoning as in the previous proof, we obtain
$$
\lim_{n\to \infty} I_{\eps_n}(u_n,A_n)
	= \frac12 \int_{\omega} \left( \bigl| \bigl(\grad'-i{A^{\rm ex}_{\perp}}'\bigr)u\bigr|^2 + |b|^2 + \frac{\kappa^2}{2} \bigl( 1 - |u|^2 \bigr)^2 \right) \, dx'.
$$
This completes the proof.
\end{proof}


\section{Supercritical Case}
\label{supersec}

\begin{theorem}[Compactness]\label{thm:compactness_super}
Let $\eps_n \to 0^+$ as $n\to+\infty$ and let $\{u_{n}, A_n - A^{\rm ex}\}_{n\in\NN} \subset H^{1}(\Omega;\C) \times \HHr$ be such that
$$
\sup_{n\in\NN}\, I_{\eps_n}(u_n,A_n) < +\infty.
$$
Then there exist a subsequence $\{\eps_n\}$ (not relabeled) such that 
\begin{align*}
& u_n \to 0  \text{ in } L^{2}(\Omega;\C), \\
& \frac{1}{\eps_n\rho_n} A_n - A^{\rm ex}_{\|} \wto 0 \text{ in } \HHr, \\\end{align*}
\end{theorem}

\begin{theorem}[$\Gamma$--limit]\label{thm:gammalim_super}
Let $(u,A) \in L^1(\Omega;\C) \times L^1(\R^3;\R^3)$.
Then
$$
\Gamma-\lim_{\eps \to 0^+} I_{\eps}(u,A)
	= \begin{cases}
	\ds \frac{\kappa^2}{4} |\Omega| 	& \text{ if } u\equiv 0 \text{ and } A = A^{\rm ex}_{\|}\\
	\infty		& \text{ otherwise.}
	\end{cases}
$$
\end{theorem}

\subsection{Compactness}

\begin{lemma}\label{lem:conv}
Let $\{f_n\}, \{g_n\} \subset L^2(\Omega;\C^2)$ be such that
\begin{align*}
& f_n \wto 0 \text{ in } L^2(\Omega;\C^2), \\
& g_n \wto g \text{ in } L^2(\Omega;\C^2).
\end{align*}
Assume further that $f_n - g_n \to 0$ in $L^2(\Omega;\C^2)$.

Then
$$
g=0.
$$
\end{lemma}

\begin{proof}[ of Theorem \ref{thm:compactness_super}]

Let $K := \sup_{n \in \NN} I_{\eps_n}(u_n,A_n) < \infty$.
Then define
\begin{align*}
& w_n := \frac{1}{\eps_n \rho_{n}} u_n, \\
& B_n := \frac{1}{\eps_n \rho_{n}} A_n, \\
& \ell_n := \grad \times B_n = \frac{1}{\eps_n \rho_{n}} h_n, \\
& e_n := \lambda_n^{-1} {(\ell_n)}_3 = {(h_n)}_3.
\end{align*}

Then
$$
\int_{\R^3} \left(\rho_n^2 \left| \ell_n' - {h^{\rm ex}}'\right|^2 + \left| {(h_n)}_3 - h^{\rm ex}_3\right|^2\right) \, dx \leq K,
$$
This implies that $\grad \times (B_n - A^{\rm ex})$ is bounded in $L^2$, so since $\div B_n = 0$ we know that 
$$
\sup_{n\in\NN} \|\grad (B_n - A^{\rm ex})\|_{L^2(\R^3;\R^{3\times 3})} < \infty,
$$
and ${(\ell_n)}_3 \to 0$ in $L^2(\R^3)$.

By Lemma~\ref{lem:GP}, we deduce that $\{B_n - A^{\rm ex}\}$ is bounded in $\HH$, thus there exists a subsequence (not relabeled) such that
$$
B_n  - A^{\rm ex}\wto B - A^{\rm ex} \quad \text{ in } \HHr.
$$
and
$$
\div B = 0 \; ,  \qquad 
\grad \times B = (h^{\rm ex}_1,h^{\rm ex}_2,0).
$$

Moreover consider $B-A^{\rm ex}_{\|}$, which satisfies $\grad \times (B-A^{\rm ex}_{\|}) = 0$. By the uniqueness in Lemma \ref{lem:GP}, we deduce that
$$
B\equiv A^{\rm ex}_{\|}.
$$

On the other hand, we know that $\{u_n\}_{n \in \NN}$ is bounded in $L^4 (\Omega;\C)$,
that $\grad B_n$ is bounded in $L^2(\Omega;\R^{3\times 3})$, and
$$
\int_{\Omega} \left( \left| \grad' u_n - i A'_n u_n\right|^2  + \left| \frac{1}{\eps_n} \partial_3 u_n - i {(A_n)}_3 u_n\right|^2 \right) \, dx \leq K,
$$
so
$$
(\eps_n \rho_n)^2 \int_{\Omega} \left( \left| \grad' \left(\frac{u_n}{\eps_n\rho_n}\right) - i B'_n u_n\right|^2  + \left| \frac{1}{\eps_n} \partial_3\left(\frac{u_n}{\eps_n\rho_n}\right) - i {(B_n)}_3 u_n\right|^2 \right) \, dx \leq K.
$$
This yields that $\{w_n\}$ is bounded 
in $H^1(\Omega;\C)$, so we may extract a further subsequence (not relabeled) such that
$$
w_n \wto w \text{ in } H^1(\Omega;\C).
$$

On the other hand, we know that $\{u_n\}$ is bounded in $L^2(\Omega;\C)$, so we can extract another subsequence (not relabeled) such that
$$
u_n \wto u \text{ in } L^2(\Omega;\C),
$$
which implies that $w=0$. So
$$
w_n \wto 0 \text{ in } H^1(\Omega;\C).
$$

We now know that
\begin{align*}
& \grad' w_n \wto 0 \text{ in } L^2(\Omega;\C^2),  \\
& iB'_n u_n \wto i {A^{\rm ex}_{\|}}' u \text{ in } L^2(\Omega;\C^2),  \\
& \grad' w_n - iB'_n u_n \to 0 \text{ in } L^2(\Omega;\C^2),
\end{align*}
so by Lemma \ref{lem:conv}, we deduce that
$$
iB'_n u_n \to 0 \text{ in } L^2(\Omega;\C^2).
$$

Since $i B'_n u_n \to i{A^{\rm ex}_{\|}}' u$ pointwise, and ${A^{\rm ex}_{\|}}' \neq 0$, we conclude that $u=0$, so we know that
$$
u_n \to 0 \text{ in } L^2(\Omega;\C).
$$
\end{proof}

\subsection{The $\Gamma$-liminf inequality}

\begin{proposition}[$\Gamma-\liminf$ inequality]

Consider sequences $\{\eps_n\} \subset \R$, $\{u_n\} \subset H^1(\Omega;\C)$, and $\{A_n - A^{\rm ex}\} \subset \HHr$ satisfying
\begin{align*}
& \eps_n \to 0^+, \\
& u_n \to 0 \text{ in } L^2(\Omega;\C), \\
& \frac{1}{\eps_n \rho_n} A_n - A^{\rm ex}_{\|} \wto 0 \text{ in } \HHr. 
\end{align*}

Then
$$
\liminf_{n \to \infty} I_{\eps_n}(u_n,A_n) 
	\geq \frac{\kappa^2}{4}|\Omega|.
$$
\end{proposition}

\begin{proof}[]

Since $u_n \to 0$ in $L^2(\Omega;\C)$, we have
$$
\liminf_{n \to \infty} I_{\eps_n}(u_n,A_n) 
	\geq \frac{\kappa^2}{4} \liminf_{n \to \infty} \int_{\Omega} \bigl( 1-|u_n|^2\bigr)^2 \, dx
	= \frac{\kappa^2}{4} |\Omega|. 
$$
This completes the proof.
\end{proof}

\subsection{The $\Gamma$-limsup inequality}

\begin{proposition}[$\Gamma-\limsup$ inequality]

Let $\{\eps_n\} \in \R$ be a sequence such that $\eps_n \to 0^+$.
Then, there exist sequences $\{u_n\} \subset H^1(\Omega;\C)$ and $\{A_n - A^{\rm ex}\} \subset \HHr$ such that
\begin{align*}
& u_n \to 0 \text{ in } L^2(\Omega;\C), \\
& \frac{1}{\eps_n \rho_n} A_n - A^{\rm ex}_{\|} \wto 0 \text{ in } H^1(\R^3;\R^3),
\end{align*}
and
$$
\lim_{n \to \infty} I_{\eps_n}(u_n,A_n) 
	= \frac{\kappa^2}{4}|\Omega|.
$$
\end{proposition}

\begin{proof}[]

First, define
\begin{align*}
& u_n(x) := 0, \\
& A_n(x) := \eps_n\rho_n A^{\rm ex}_{\|} + A^{\rm ex}_{\perp}.
\end{align*}

Then
$$
B_n -A^{\rm ex}_{\|}:= \frac{1}{\eps_n \rho_n} A_n - A^{\rm ex}_{\|}= \frac{A^{\rm ex}_{\perp}}{\eps_n\rho_n} \wto 0 \text{ in } \HHr.
$$

Moreover,
$$
I_{\eps_n}(u_n,A_n)
	= \frac{\kappa^2}{4} |\Omega| .
$$
This completes the proof.
\end{proof}


\bibliographystyle{amsalpha}
\addcontentsline{toc}{section}{References}
\bibliography{refs}


\end{document}